\documentclass[11pt,reqno]{amsart}

\usepackage[utf8]{inputenc}
\usepackage[T1]{fontenc}
\usepackage[english]{babel}
\usepackage{tikz}
\usepackage{csquotes}

\usepackage{amsmath, amssymb, amsfonts}
\usepackage{enumerate}
\usepackage[margin=1.2in]{geometry}

\usepackage[colorlinks, citecolor=blue, linkcolor=black, urlcolor=blue]{hyperref}

\usepackage{graphicx}

\newcommand{\N}{\mathbb{N}}

\newcommand{\R}{\mathbb{R}}
\newcommand{\E}{\mathbb{E}}
\newcommand{\Z}{\mathbb{Z}}

\DeclareMathOperator*\uplim{\overline{lim}}
\newcommand{\Int}[2]{%
	\mathchoice%
	{{\displaystyle\int_{#1}^{#2}}}%
	{{\displaystyle\int_{#1}^{#2}}}%
	{\int_{#1}^{#2}}%
	{\int_{#1}^{#2}}%
}

\newcommand{\f}[2]{\frac{#1}{#2}}

\usepackage{ifthen}
\newcommand{\abs}[2][9]{%
	\ifthenelse{#1 = 0}{#2}{}%
	\ifthenelse{#1 = 1}{|#2|}{}%
	\ifthenelse{#1 = 2}{\big|#2\big|}{}%
	\ifthenelse{#1 = 3}{\Big|#2\Big|}{}%
	\ifthenelse{#1 = 4}{\bigg|#2\bigg|}{}%
	\ifthenelse{#1 = 5}{\Bigg|#2\Bigg|}{}%
	\ifthenelse{#1 = 9}{\left|#2\right|}{}%
}

\newcommand{\norme}[2][1]{%
	\ifthenelse{#1 = 0}{#2}{}%
	\ifthenelse{#1 = 1}{\Arrowvert #2\Arrowvert}{}%
	\ifthenelse{#1 = 2}{\big\Arrowvert#2\big\Arrowvert}{}%
	\ifthenelse{#1 = 3}{\Big\Arrowvert#2\Big\Arrowvert}{}%
	\ifthenelse{#1 = 4}{\bigg\Arrowvert#2\bigg\Arrowvert}{}%
	\ifthenelse{#1 = 5}{\Bigg\Arrowvert#2\Bigg\Arrowvert}{}%
	\ifthenelse{#1 = 9}{\left\Arrowvert#2\right\Arrowvert}{}%
}

\newcommand{\petito}[1][]{%
	\ifthenelse{%
		\equal{#1}{}%
	}{%
		\text{o}%
	}{%
		\underset{#1}{\text{o}}%
	}%
}

\newcommand{\grando}[1][]{%
	\ifthenelse{%
		\equal{#1}{}%
	}{%
		\text{O}%
	}{%
		\underset{#1}{\text{O}}%
	}%
}

\numberwithin{equation}{section}

\theoremstyle{plain}
\newtheorem{theorem}{Theorem}[section]
\newtheorem{lemma}[theorem]{Lemma}
\newtheorem{proposition}[theorem]{Proposition}
\newtheorem{corollary}[theorem]{Corollary}

\theoremstyle{definition}
\newtheorem{definition}[theorem]{Definition}

\theoremstyle{remark}

\title[Dynamical interface above a hard wall]{Dynamical interface above a hard wall \\
and reflected SPDE on the half-line}

\author{Pierre Faugère}
\address{Universit\'e Paris Cit\'e, CNRS, Laboratoire de Probabilit\'es, Statistique et Mod\'elisation, UMR 8001, F-75205 Paris, France}
\email{faugere@lpsm.paris}
\author{Cyril Labbé}
\address{Universit\'e Paris Cit\'e, Laboratoire de Probabilit\'es, Statistique et Mod\'elisation, UMR 8001, F-75205 Paris, France and Institut Universitaire de France (IUF).}
\email{clabbe@lpsm.paris}

\begin{document}

\begin{abstract}
	We consider a dynamical random interface on the infinite lattice $\N$ evolving according to a "corner flip" dynamic above a hard wall, with an additional pinning at the origin. 
	We study the stationary fluctuations under a diffusive scaling and prove convergence in law towards the solution of an SPDE of Nualart-Pardoux's type, namely the Reflected Stochastic Heat Equation on the half-line. We also obtain that the law of the 3-dimensional Bessel process is an invariant measure for this SPDE.
\end{abstract}

\maketitle
\tableofcontents

\section{Model and main results}

\subsection{Discrete dynamic above a hard wall}
We consider a Markov process $(h_t)_{t \geq 0}$ 
with state space 
\begin{equation}
    \mathcal{X}:=\left\lbrace h \in \N^\N \;: \; \forall n \in \N \quad  \abs{h(n+1)-h(n)}=1; \quad h(0)=0  \right\rbrace
\end{equation}
solution of the following system of stochastic differential equations
\begin{equation} \label{semimartingale eq poisson}
    \left\lbrace 
    \begin{array}{ll}
         &dh_t(n)=\Delta h_t(n) 1 \left\lbrace h_t(n) +\Delta h_t(n) \geq 0 \right\rbrace  \: dN_t(n) \qquad \forall n \in \N^*:= \N \backslash\{0\} \\
         &h_0=\zeta \in \mathcal{X}
    \end{array}
    \right.
\end{equation}
\noindent where we used the notation $\Delta h(n):=h(n+1)+h(n-1)-2h(n) \in \{ -2, 0, 2 \}$ for the discrete Laplacian and where $(N_\cdot(n))_{n \in \N}$ is a family of independent Poisson processes of parameter one.
The process $(h_t)_{t \geq 0}$ corresponds to a random interface evolving according to the "corner flip" dynamic and constrainded to remain above a hard wall at height zero, see Figure \ref{figure} for a graphical explanation. As we will see later on, the law $\pi$ on $\mathcal{X}$ of a symmetric random walk $(X_n)_{n \in \N}$ starting from zero and conditioned to remain non-negative is an invariant distribution for the process $(h_t)_{t \geq 0}$. We will work under the \textit{diffusive scaling}, meaning that for $\epsilon \in (0,1]$ we will consider the rescaled process $h^\epsilon$
\begin{equation}
    \forall x \in \epsilon\N, \quad \forall t\geq 0, \qquad h^\epsilon_t(x):=\sqrt{\epsilon} h_{\epsilon^{-2}t}(\epsilon^{-1}x).
\end{equation}
For fixed $t \geq 0$, we shall consider $h^\epsilon_t$ as a continuous function on $[0,\infty)$ by linear interpolation at its values on the lattice $\epsilon \N$. When viewed this way, $h^\epsilon$ is a random element of the Skorokhod space $D([0,\infty), C([0,\infty)))$, where $C([0,\infty))$ denotes the space of real continuous function on $[0,\infty)$ endowed with the local uniform topology. For each $\epsilon \in (0,1]$, the law $\pi^\epsilon$ on $C([0,\infty))$ obtained as the pushforward of $\pi$ through the above rescaling is invariant for $(h^\epsilon_t)_{t \geq 0}$. Moreover, by a generalization of Donsker's invariance principle \cite{bryndoney}, the family of stationary laws converges in the limit $\epsilon \to 0$ towards the law of the 3-dimensional Bessel process. Our aim is to study the scaling limit of the discrete dynamic -- starting from equilibrium -- as $\epsilon$ goes to zero, and to describe the limiting object in the continuum, providing a sort of \textit{dynamical invariance principle}.\\

This model presents two main features: the presence of the wall constraint, and the fact that the interface lives on an unbounded spatial domain. Without the wall constraint and in unbounded spatial domain, convergence of fluctuations towards the additive stochastic heat equation is known. The presence of the wall is expected to induce a reflection term in the stochastic PDE obtained in the continuum. With the wall constraint but on a segment, this convergence towards Nualart-Pardoux's equation was proven in \cite{etheridge2015scaling}. 
In other words, taken separately, each of these two problems has been solved. The aim of this paper is to overcome both difficulties at the same time.\\

Let us also mention some works on related topics. First, the discrete dynamic above a hard wall on the whole lattice $\Z$ was studied in \cite{dunlopferrari}, where it is proved that the model exhibits a phenomenon of entropic repulsion. Second, a convergence result towards the solution of Nualart-Pardoux's reflected SPDE on a segment for a system of coupled oscillators driven by SDEs of Skhorokhod type was proven in \cite{funakiolla}. The approximating model from \cite{funakiolla} differs from the one from \cite{etheridge2015scaling} or the one from this present work since the interface takes continuous values rather than discrete ones.

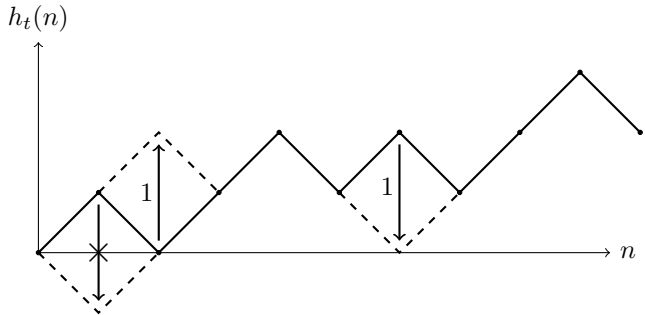
\begin{figure} \label{figure}
\centering
\begin{tikzpicture}[scale=0.8, every node/.style={font=\small}]

\draw[->] (0,0) -- (9.5,0) node[right] {$n$};
\draw[->] (0,0) -- (0,3.5) node[above] {$h_t(n)$};

\coordinate (A) at (0,0);
\coordinate (B) at (1,1);
\coordinate (C) at (2,0);
\coordinate (D) at (3,1);
\coordinate (E) at (4,2);
\coordinate (F) at (5,1);
\coordinate (G) at (6,2);
\coordinate (H) at (7,1);
\coordinate (I) at (8,2);
\coordinate (J) at (6,2);
\coordinate (K) at (1,-1);
\coordinate (L) at (2,2);
\coordinate (M) at (5,3);
\coordinate (N) at (9,3);
\coordinate (O) at (6,0);
\coordinate (P) at (10,2);

\draw[thick] (A) -- (B) -- (C) -- (D) -- (E) -- (F) -- (G) -- (H) -- (I) -- (N) -- (P);
\draw[thick, dashed] (A) -- (K) -- (C);
\draw[thick, dashed] (B) -- (L) -- (D);
\draw[thick, dashed] (F) -- (O) -- (H);

\foreach \pt in {A,B,C,D,E,F,G,H,I,J,N,P}
  \filldraw[black] (\pt) circle (1pt);

\draw[->, thick] (2,0.2) -- (2,1.8);
\node at (1.8,1) {$1$};

\draw[->, thick] (1,0.8) -- (1,-0.8);
\node at (1,0) {\Large $\times$};

\draw[->, thick] (6,1.8) -- (6,0.2) ;
\node at (5.8,1.1) {$1$};

\end{tikzpicture}
\caption{Graphical representation of the jump rates for the discrete dynamic. To each is site is associated a random Poisson clock of parameter one, independent from those of the other sites. Every time a clock rings, if the corresponding site forms a corner, we flip it, except if the flipped interface takes negative values. Here, non-crossed arrows represent possible transitions with their associated rate, while the crossed arrow represents a forbidden transition.}
\end{figure}

\subsection{Stochastic PDE with reflection}
Let us introduce in this paragraph the \textit{stochastic heat equation with reflection} on the half-line which will be obtained after taking the scaling limit, that is to say in the limit $\epsilon \to 0$. We fix a \textit{cylindrical Wiener process}, that is an $\mathcal{S}'([0,\infty))$-valued random process $(W_t)_{t \geq 0}$ such that for all $\varphi \in C^\infty_c([0,\infty))$, $(W_t(\varphi))_{\geq 0}$ is a Brownian motion with variance $\norme{\varphi}_{L^2([0,\infty))}^2$. Note that the derivative in time, $\dot{W}$, is then a space-time white noise on $[0,\infty) \times [0,\infty)$. We consider the following equation on the pair $(u,\eta)$
\begin{equation} \label{Nualar Pardoux eq}
    \left\lbrace
    \begin{array}{ll}
         &\partial_t u (t,x)= \partial_{xx}^2 u (t,x) + \sqrt{2} \dot{W}_t(x) + \eta (dt,dx) \qquad \forall t \geq 0, \quad \forall x \in [0,\infty) \\
         &u(0,x) =u_0, \; u(t,0)=0 \qquad \forall t \geq 0 \\
         &u \geq 0, \; d\eta \geq 0, \; \int u \; d\eta =0
    \end{array}
    \right.
\end{equation}
starting from a fixed initial condition $u_0 \geq 0$ with $u_0 \in \mathcal{C}_\rho$ for some $\rho >0$,  where
\begin{equation*}
    \mathcal{C}_\rho:= \left\lbrace f \in C([0,\infty)) \:: \:  f(0)=0, \; \displaystyle{\sup_{x \in [0,\infty)}} \abs{f(x)} e^{-\rho x} < \infty \right\rbrace.
\end{equation*}
Let us be more precise about the notion of solution for the above equation.
\begin{definition} \label{definition de l'eq reflechie}
We say that a pair $(u ,\eta)$ is a solution to \eqref{Nualar Pardoux eq} if
\begin{enumerate}[(i)]
    \item $\left(u(t,\cdot) \; , \; t \geq 0 \right)$ is a continuous $\mathcal{C}_\rho$-valued stochastic process.
    \item $u \geq 0$.
    \item $\eta$ is a random measure on $[0,\infty) \times (0,\infty)$ such that for all compact $[0,T]\times [a,b] \subseteq [0,\infty) \times (0,\infty)$, $\eta \left( [0,T] \times [a,b] \right) < \infty$.
    \item For all $t \geq 0$ and $\varphi \in C^\infty_c( (0,\infty))$
    \begin{multline} \label{Nualart Pardoux variational formulation}
        \left\langle u(t,\cdot), \varphi \right\rangle = \left\langle u_0, \varphi \right\rangle + \Int{0}{t} \left\langle u(s,\cdot) , \varphi '' \right\rangle \: ds
        +\sqrt{2} W_t(\varphi) + \Int{0}{t} \Int{0}{\infty} \varphi(x) \: \eta (ds,dx).
    \end{multline}
    \item \label{support condition} $\Int{}{} u \: d \eta =0$ or equivalently the support of $\eta$ is contained in the zero level set of $u$.
\end{enumerate}
\end{definition}

\noindent Here $\langle \cdot, \cdot \rangle$ denotes the inner product in $L^2([0,\infty), dx)$. The specificity of this equation lies in the measure $\eta$ which imposes a reflection condition. Indeed, the presence of $\eta$ "forces" the solution to remain non-negative, while the support condition \eqref{support condition} ensures that the measure only acts when $u(t,x)=0$, so that intuitively,  $u$ solves the classical stochastic heat equation whenever $u(t,x)>0$. Reflected stochastic PDEs of this type were first introduced by Nualart and Pardoux, who studied in \cite{Nualart1992WhiteND} the case where the spatial domain is the segment $[0,1]$ with Dirichlet boundary conditions, and proved strong existence and uniqueness for the problem. Our case is different since we consider $[0,\infty)$ as a spatial domain with pinning at the origin, for which strong existence and uniqueness was proved in \cite[Theorem 2.6]{hambly2019reflected}.\\
Finally, note that the reflection measure $\eta$ is far from being a trivial object. Let us illustrate this with some properties of its support proven by Dalang, Mueller and Zambotti in \cite{zambottihitting} in the case where the spatial domain is the segment $[0,1]$. For every fixed $t >0$, almost surely for every $x \in (0,1)$, $u(t,x)>0$. Consequently, by the support condition, for every $t \geq 0$, almost surely $\mathcal{Z}(t)= \emptyset$, where $\mathcal{Z}(t):= \left\lbrace x \in (0,1) \; : \; (t,x) \in \textrm{supp}(\eta) \right\rbrace$. This means that the reflection measure only acts at exceptional times $t \geq 0$, but still impacts globally the behavior of the solution. More precisely, with positive probability, there exists at least an exceptional time $t > 0$, such that the cardinality of $\mathcal{Z}(t)$ is at least three. On the other hand, almost surely, at all times  $t>0$, the cardinality of $\mathcal{Z}(t)$ is upper bounded by four. Let us also mention that $\eta$ can be interpreted as a local time of $u$ but not exactly in the classical sense, see \cite{zambottioccupation}.

\subsection{From the semimartingale equation to a semi-discrete PDE} To understand the connection between the discrete dynamic and the reflected stochastic PDE, let us display a convenient rewriting of equation \eqref{semimartingale eq poisson}. First, under the diffusive scaling, equation \eqref{semimartingale eq poisson} becomes
\begin{equation} \label{rescaled semimartingale eq}
    \forall x \in \epsilon \N^* \qquad dh^\epsilon_t(x):= 
    \Delta^\epsilon h^\epsilon_t (x) 1 \left\lbrace h_t^\epsilon(x) +\Delta^\epsilon h^\epsilon_t (x) \geq 0  \right\rbrace
    \: dN^\epsilon_t(x)
\end{equation}
where $(N^\epsilon_\cdot(x))_{x \in \epsilon \N^*}$ is the family of Poisson processes defined by $N^\epsilon_t(x):=N_{\epsilon^{-2}t}(\epsilon^{-1}x)$ for $x \in \epsilon \N^*$, and where we used the notation $\Delta^\epsilon f(x):= f(x+\epsilon)+f(x-\epsilon)-2f(x)$.
Second, to make the approximate white noise term appear, we split in \eqref{rescaled semimartingale eq} the Poisson term into a martingale and a drift term by considering the family of martingales $(M^\epsilon_\cdot(x))_{x \in \epsilon \N^*}$ defined by $M^\epsilon_t(x):=N^\epsilon_t(x)-\epsilon^{-2} t$  for $x \in \epsilon \N^*$, obtaining
\begin{multline} \label{second semimartingale eq}
    \forall x \in \epsilon \N^* \qquad 
    dh^\epsilon_t(x)=
    \f{1}{\epsilon^2} \Delta^\epsilon h^\epsilon_t(x) 1 \left\lbrace h^\epsilon_t(x) +\Delta^\epsilon h^\epsilon_t(x) \geq 0 \right\rbrace \: dt  \\ + \Delta^\epsilon h^\epsilon_t(x) 1 \left\lbrace h^\epsilon_t(x) +\Delta^\epsilon h^\epsilon_t(x) \geq 0 \right\rbrace \: dM^\epsilon_t(x).
\end{multline}
This leads us to introduce the \textit{discrete noise}
\begin{equation}
    W^\epsilon_t(dx):= 
    \f{\epsilon }{\sqrt{2}} \sum_{k \in \epsilon \N^*} \Int{0}{t} \Delta^\epsilon h^\epsilon_s(k) 1 \left\lbrace h_s^\epsilon (k) +\Delta^\epsilon h^\epsilon_s (k) \geq 0 \right\rbrace \delta_k (dx) \, dM_s^\epsilon(k)
\end{equation}
which defines an $S'([0,\infty))$-valued random process $(W^\epsilon_t)_{t \geq 0}$.
Third, to make the reflection term appear, we split in \eqref{second semimartingale eq} the drift term appropriately. Using the fact that for all $x \in \epsilon \N^*, t \geq 0$, $\Delta^\epsilon h_t(x) \in \left\lbrace -2\sqrt{\epsilon},0, 2\sqrt{\epsilon} \right\rbrace$, we obtain
\begin{multline}  \label{third semimartingale equation}
     \forall x \in \epsilon \N^* \qquad dh^\epsilon_t(x)=
    \f{1}{\epsilon^2}  \Delta^\epsilon h^\epsilon_{t} (x) dt 
    +\f{2 \sqrt{\epsilon}}{\epsilon^2} 1 \left\lbrace h^{\epsilon}_{t}(x)+\Delta^\epsilon h^\epsilon_{t}(x) < 0 \right\rbrace dt
    +\f{\sqrt{2}}{\epsilon} dW_t^\epsilon (dx) .
\end{multline}
This leads us to introduce the \textit{discrete reflection measure}
    \begin{equation}
        \eta^\epsilon(dt,dx):= \f{2}{\sqrt{\epsilon}} \sum_{k \in \epsilon \N^*}{} 1 \left\lbrace h_t^\epsilon (k) +\Delta^\epsilon h^\epsilon_t (k) <0 \right\rbrace \delta_k (dx) \, dt
    \end{equation}
which is a random element of the subspace $\mathbb{M}$ of the space of Borel measures on $ [0,\infty) \times [0,\infty)$ defined by
\begin{equation} \label{space M of measures}
    \mathbb{M} := \left\lbrace \nu \; : \; \forall T,A \geq 0  \qquad \Int{[0,T] \times [0,A]}{} x \nu(dt,dx) < \infty \right\rbrace.
\end{equation}
which we endow with the vague topology. Fourth, testing the semimartingale equation \eqref{third semimartingale equation} against some $\varphi \in C^\infty_c( [0,\infty))$, and using the discrete noise and reflection term previously introduced, it becomes
\begin{multline} \label{variational semimartingale equation}
    \langle h^\epsilon_t, \varphi \rangle_\epsilon  = \langle h^\epsilon_0, \varphi \rangle_\epsilon +
     \Int{0}{t} \f{1}{\epsilon^2} \langle \Delta^\epsilon h^\epsilon_s, \varphi(s,\cdot) \rangle_\epsilon \: ds
    + \sqrt{2} W^\epsilon_t(\varphi)
    + \Int{0}{t} \Int{0}{\infty} \varphi (x) \: \eta^\epsilon (ds,dx)
\end{multline}
where we used the notation $\left\langle \cdot, \cdot \right\rangle_\epsilon:= \epsilon \langle \cdot , \cdot \rangle_{l^2(\epsilon \N)}$. Let us draw the reader's attention on the parallel between, on the one hand, the weak formulation \eqref{Nualart Pardoux variational formulation} of the reflected SPDE in the continuous setting and, on the other hand, the semi-discrete equation \eqref{variational semimartingale equation} for the random interface.

\subsection{Main results}

We may now state our main result on the convergence of the stationary fluctuations of the discrete interface model towards a reflected stochastic PDE. In the following statement, $\zeta$ will denote a $\pi$-distributed random variable, independent of the collection of Poisson processes $(N_\cdot(n))_{n \in \N}$.

\begin{theorem} \label{main theorem}
    Consider the random process $(h_t)_{t \geq 0}$ defined by \eqref{semimartingale eq poisson} and starting from an initial condition $\zeta$ distributed according to the stationary measure $\pi$. Consider the associated sequence $(h^\epsilon, W^\epsilon, \eta^\epsilon)_{\epsilon \in (0,1]}$ of  $D([0,\infty),\mathcal{C}_\rho)\times D( [0,\infty), \mathcal{S}'([0,\infty ) )) \times \mathbb{M}$-valued random variables, the latter space being endowed with the product topology. Then $$(h^\epsilon, W^\epsilon, \eta^\epsilon) \xrightarrow[\epsilon \to 0]{\mathcal{L}} (u,W,\eta)$$ where
    \begin{enumerate}[(i)]
        \item $W$ is a cylindrical Wiener process,
        \item $(u,\eta)$ is the solution of the reflected stochastic PDE \eqref{Nualar Pardoux eq} starting from a random initial condition $u_0$ distributed according to the law of the 3-dimensional Bessel process, and independent of $W$.
    \end{enumerate}
\end{theorem}

\begin{corollary}\label{Corollary}
    The law of the 3-dimensional Bessel process starting from zero is invariant for the reflected stochastic PDE \eqref{Nualar Pardoux eq}.
\end{corollary}

\noindent The general strategy of the proof is to show that equation \eqref{variational semimartingale equation} becomes in the limit \eqref{Nualart Pardoux variational formulation}. For this we prove individually tightness of the sequences $(h^\epsilon)_{\epsilon \in (0,1]}$, $(W^\epsilon)_{\epsilon \in (0,1]}$, and $(\eta^\epsilon)_{\epsilon \in (0,1]}$ and prove that any limit point $(u,\eta,W)$ as $\epsilon \to 0$ is solution to \eqref{Nualart Pardoux variational formulation}. Let us comment more precisely on the proof techniques. For tightness of both $(h^\epsilon)_{\epsilon \in (0,1]}$ and $(W^\epsilon)_{\epsilon \in (0,1]}$, the proofs rely on two main ingredients: static and dynamical properties. Static properties are quantitative estimates related to the invariance principle for the random walk conditioned to remain non-negative, while dynamical estimates leverage the martingale structure of the dynamic, using a double Burkholder-Davies-Gundy inequality technique inspired by \cite{giacomin} and \cite{etheridge2015scaling}. However, for the tightness of $(h^\epsilon)_{\epsilon \in (0,1]}$ more specifically, there is an obstacle coming from the fact that the semimartingale equation \eqref{second semimartingale eq} governing the dynamic comes with a reflection term, delicate to control a priori. We overcome this using \textit{Lyons-Zheng's decomposition} \cite{lyonszheng}, which relies on the reversibility of the dynamic and stationarity, in order to reduce the problem to bounds on moments of increments of some martingale, for which we can then apply the double BDG technique aforementioned. Let us comment on the challenges specific to the infinite volume case that we had to overcome, which are new compared to \cite{etheridge2015scaling}. The main difficulty comes from the fact that in order to be solution of the reflected stochastic PDE \eqref{Nualar Pardoux eq}, we need some control at infinity on the spatial growth of the solution at each fixed time $t \geq 0$. This requires estimates for the Sobolev norm of the discretization (uniform in $\epsilon \in (0,1]$) in infinite volume.

\subsection*{Acknowledgements} The work of C.~L. was partially supported by the ANR project RANDOP ANR-24-CE40-3377, and by the Institut Universitaire de France.

\section{Preliminaries: generator and invariant measure} \label{preliminary section}
\subsection{Generator and martingale problem}
As we are in infinite volume, we recall in this paragraph some elements of the theory enabling us to construct the evolving random interface that we consider. More precisely, we want to show that the collection $\lbrace P^\zeta , \; \zeta \in \mathcal{X} \rbrace$ of laws on $D([0,\infty),\mathcal{X})$ induced by $(h_t)_{t \geq 0}$ is a Feller process and identify its generator. First, consider the operator $L$ defined by
\begin{equation} \label{generator}
     Lf(h):= \sum_{n \in \N^*} 1 \left\lbrace h+\Delta h(n)\delta_n \geq 0 \right\rbrace \left[ f(h+\Delta h(n)\delta_n)-f(h) \right]
\end{equation}
for any cylindrical function $f:\mathcal{X} \xrightarrow[]{} \R$. 
As defined, $L$ is a Markov pregenerator and by \cite[Theorem I, 3.9]{liggett1985interacting} its closure $\bar{L}$ is a Markov generator. Second, let us relate $L$ to our dynamic using a martingale problem.
For any $n \in \N^*$, consider the cylindrical function $p_n : h \mapsto h(n)$. A direct computation shows that $Lp_n(h)= \Delta h(n) 1 \left\lbrace h(n)+\Delta h(n) \geq 0 \right\rbrace$. Thus, denoting $M_t(n):= N_t(n) -t$ the family of compensated Poisson processes, $h$ satisfies
\begin{align*}
    p_n(h_t)-p_n(h_0)-\Int{0}{t} Lp_n(h_s)\: ds &= p_n(h_t)-p_n(h_0)-\Int{0}{t}  \Delta h_s(n) 1 \left\lbrace h_s(n)+\Delta h_s(n) \geq 0 \right\rbrace \: ds \\
    &= \Int{0}{t} \Delta h_s(n) 1 \left\lbrace h_s(n)+\Delta h_s(n) \geq 0 \right\rbrace \: dM_s(n).
\end{align*} This shows that the process
\begin{equation} \label{martingale pb}
    \left( p_n(h_t)-p_n(h_0)-\Int{0}{t} Lp_n(h_s)\: ds \right)_{t \geq 0} \textrm{ is a martingale.}
\end{equation}
Since any cylindrical function is a linear combination of the functions $(p_n)_{n \in \N}$, \eqref{martingale pb} extends to any cylindrical function. In other words, for each $\zeta \in \mathcal{X}$, $P^\zeta$ satisfies the martingale problem associated to $L$ and $\zeta$. But by \cite[Theorem I, 5.2]{liggett1985interacting}, the Feller process generated by $\bar{L}$ is the unique solution of the martingale problem associated to $L$. This proves that $\lbrace P^\zeta , \; \zeta \in \mathcal{X} \rbrace$ is the Feller process generated by $\bar{L}$.

\subsection{Reversible measure of the dynamic}
 In this paragraph we introduce the simple random walk conditioned to remain nonnegative and show that it is invariant for the dynamic \eqref{semimartingale eq poisson}. Consider the symmetric simple random walk starting from zero $(X_n)_{n \in \N}$ on the canonical space $(\mathcal{X}, \mathcal{F},P)$. Then the law $\pi$ of the \textit{simple random walk conditioned to remain nonnegative} can be defined by
\begin{equation} \label{egalite caravenna}
    \pi ( B ) := E \left[ (X_n +1 ) 1_B 1_{C_n} \right] \qquad \forall B \in \sigma (X_1, \cdots , X_n )
\end{equation}
where $C_n:= \lbrace X_1 \geq 0, \cdots, X_n \geq 0 \rbrace$. This terminology is justified by the following fact proved in \cite[Theorem 1]{bertoin1994conditioning}
\begin{equation*}
    \pi (\cdot) = \displaystyle{\lim_{n \to \infty} P \left( \cdot \: | \: C_n  \right)}
\end{equation*}
As defined in \eqref{egalite caravenna}, $\pi$ is obtained by Doob $h$-transform of the simple symmetric random walk via the function $h:x \mapsto x+1$ harmonic with respect to the transition semigroup of the simple symmetric random walk and which vanishes at $-1$. Then under $\pi$ the process $(X_n)_{n \in \N}$ is Markovian with state space $\N$, characterized by the following probability transitions \cite{bertoin1994conditioning}
\begin{align}
    \forall k \in \N, \qquad p_{k,k+1}=\f{k+2}{2(k+1)}, \qquad
    p_{k,k-1}=\f{k}{2(k+1)}
\end{align}

\begin{lemma}(Reversibility of $\pi$){\large \textbf{.}}
    For any cynlindrical functions $f,g : \mathcal{X} \longrightarrow \R$
    \begin{equation} \label{reversibility eq}
        \Int{\mathcal{X}}{} Lf(h)g(h) \: \pi (dh) = \Int{\mathcal{X}}{} f(h)Lg(h) \: \pi (dh)
    \end{equation}
	As a consequence the dynamics is reversible with respect to $\pi$.
\end{lemma}

\begin{proof}
    Let $N \in \N$ large enough such that $f,g$ only depend on the sites $\lbrace 0,\cdots,N \rbrace$. Then consider the restriction map
    \begin{equation*}
    \begin{array}{rl}
         T_{N+1} : \mathcal{X} &\longrightarrow \mathcal{X}_{N+1}  \\
         h&\longmapsto h_{| \lbrace 0, \cdots, N +1 \rbrace }
    \end{array}
    \end{equation*}
    where $\mathcal{X}_N:= \left\lbrace h \in \N^{\lbrace 0,\cdots, N-1 \rbrace} \;: \; \forall n \in \lbrace 0,\cdots, N \rbrace \quad  \abs{h(n+1)-h(n)}=1, \quad h(0)=0  \right\rbrace$.
    The important fact is that by \eqref{egalite caravenna}, two paths of $\mathcal{X}_{N+1}$ that end up at the same height at step $N+1$ are given the same weight under $T_{N+1} \circ \pi$. In particular, 
    $$\forall h \in \mathcal{X}_{N+1}, \quad \forall \, 1 \leq n \leq N \qquad T_{N+1} \circ \pi ( \left\lbrace h+\Delta h (n) \delta_n \right\rbrace )=T_{N+1} \circ \pi (\{ h \})$$
    With this property at hand, \eqref{reversibility eq} follows from a straightforward computation. Then, the fact that \eqref{reversibility eq} implies $\pi$ is reversible is a consequence of \cite[Theorem I, 5.3]{liggett1985interacting}.
\end{proof}

\subsection{The simple random walk conditioned to remain nonnegative}

In this paragraph we state and prove several properties related to the invariant measure $\pi$ which will be useful later.

\begin{lemma}(Transience){\large \textbf{.}} For any $\varphi \in \mathcal{S}([0,\infty])$, the following convergence holds \label{lemme 1 annexe}
    \begin{equation}
        \f{1}{N} \sum_{n \in \N} 1 \left\lbrace X_n=k \right\rbrace \varphi \left( \f{n}{N} \right) \xrightarrow[N \to \infty]{ } 0
    \end{equation}
    almost-surely and in $L^1(\pi)$.
\end{lemma}

\begin{proof}
    For $N \in \N^*$, let $Z_n:=\f{1}{N} \sum_{n \in \N} 1 \left\lbrace X_n=k \right\rbrace \varphi \left( \f{n}{N} \right)$. By \cite[Theorem 3.1]{lamperti1960criteria}, the process $(X_n)_{n \in \N}$ is transient. Now it follows from the transience of $X$ and the fact that $\varphi$ is bounded that $Z_n$ converges to zero almost surely. Now let $K_\varphi:= \sup \left\lbrace  \f{1}{N} \, \sum_{n \in \N} \varphi \left( \f{n}{N} \right) \: : \: N \in \N \right\rbrace < \infty$ as $\varphi \in \mathcal{S}([0,\infty])$. Then almost surely, for all $n \in \N$, $\abs{Z_n} \le K_\varphi$.
    The $L^1(\pi)$ convergence then holds by dominated convergence. 
\end{proof}

\noindent For $n \in \N^*$ we define the discrete Laplacian by 
\begin{equation}
    \Delta X_n:= X_{n+1} + X_{n-1} -2X_n
\end{equation}

\begin{lemma}(Average number of corners){\large \textbf{.}} For any $\varphi \in \mathcal{S}([0,\infty))$, the following convergence holds \label{lemme 2 annexe}
    \begin{equation} \label{eq lemme 2 annexe}
        \f{1}{N} \sum_{n \in \N} 1 \left\lbrace \Delta X_n \neq 0 \right\rbrace \varphi \left( \f{n}{N} \right) \xrightarrow[N \to \infty]{ } \f{1}{2} \int_{0}^\infty \varphi (x) \: dx
    \end{equation}
    in $L^1(\pi)$.
\end{lemma}

\begin{proof}
    First, note that given a sequence $(B_n)_{n \in \N}$ of i.i.d.~Bernoulli random variables of parameter $1/2$, by a straightforward computation of the expectation and the variance, we have
    \begin{equation*}
        \f{1}{N} \sum_{n \in \N} B_n \varphi \left( \f{n}{N} \right) \xrightarrow[N \to \infty]{L^2} \f{1}{2} \int_{0}^\infty \varphi (x) \: dx.
    \end{equation*}
    Therefore, if we replaced in the statement $(X_n)_{n \in \N}$ by a simple symmetric random walk $(S_n)_{n \in \N}$ then \eqref{eq lemme 2 annexe} would hold. Indeed, as $\left( 1 \left\lbrace \Delta S_{2n+1} \neq 0   \right\rbrace \right)_{n \in \N}$ and $\left( 1 \left\lbrace \Delta S_{2n} \neq 0   \right\rbrace \right)_{n \in \N}$ are families of i.i.d. Bernoulli random variables of parameter $1/2$, we have
    \begin{equation} \label{LGN}
        \left\lbrace
        \begin{array}{rl}
             &\f{1}{N} \sum_{n \in \N} 1 \left\lbrace \Delta S_{2n+1} \neq 0  \right\rbrace \varphi \left( \f{n}{N} \right) \xrightarrow[N \to \infty]{L^1(\pi)} \f{1}{4} \int_{0}^\infty \varphi (x) \: dx  \\[3mm]
             &\f{1}{N} \sum_{n \in \N} 1 \left\lbrace \Delta S_{2n} \neq 0   \right\rbrace \varphi \left( \f{n}{N} \right) \xrightarrow[N \to \infty]{L^1(\pi)} \f{1}{4} \int_{0}^\infty \varphi (x) \: dx.
        \end{array}
        \right.
    \end{equation}
    So that
    \begin{equation*}
        \f{1}{N} \sum_{n \in \N} 1 \left\lbrace \Delta S_{n} \neq 0  \right\rbrace \varphi \left( \f{n}{N} \right) \xrightarrow[N \to \infty]{L^1(\pi)} \f{1}{2} \int_{0}^\infty \varphi (x) \: dx.
    \end{equation*}
    We want to prove that each of the two convergences of \eqref{LGN} remains in force with $X$ in place of $S$. For simplicity, we present the details only for the first convergence. The strategy is to build a coupling between $X$ and $S$ such that
    \begin{equation} \label{domination stochastique marche conditionnée}
        \f{1}{N} \sum_{n \in \N} \abs{  
        1 \left\lbrace \Delta X_{2n+1} \neq 0   \right\rbrace
        - 1 \left\lbrace \Delta S_{2n+1} \neq 0   \right\rbrace }
        \leq \f{1}{N} \sum_{n \in \N} 1 \left\lbrace X_{2n+1} = 0   \right\rbrace
    \end{equation}
    then Lemma \ref{lemme 1 annexe} shows that the right hand side of \eqref{domination stochastique marche conditionnée} goes to zero in $L^1(\pi)$ as $N \to \infty$, so together with \eqref{LGN} it enables us to conclude. We now establish the coupling. Set $\mathcal{W}:=\left\lbrace \wedge, \vee, - \right\rbrace$ and let us introduce the family $(W_n)_{n \in \N}$ of i.i.d. $\mathcal{W}$-valued random variables such that for all $n \in \N$
\begin{align*}
    \pi(W_n=\wedge)=\pi(W_n=\vee)=\f{1}{4} \quad \textrm{and} \quad
    \pi(W_n=-)=\f{1}{2}.
\end{align*}
    Additionaly, let us take a family $(B_{n,k})_{n \in \N, k\in \N\backslash{0,1}}$ of independent Bernoulli random variables independent of $W$, such that 
    $$B_{n,k} \sim \mathcal{B} \left(\f{p_{k,k+1}p_{k+1,k+2}}{p_{k,k+1}p_{k+1,k+2}+p_{k,k-1}p_{k-1,k-2}} \right)\;,\quad n\ge 0\;,\quad k\ge 2$$
    We can then construct inductively the process $(X_n)_{n \in \N}$ by setting
    $X_0 := 0$ and for $n \geq 0$
    \begin{align*}
        (X_{2n+1},X_{2n+2}):=&1 \left\lbrace W_n= \wedge \right\rbrace (X_{2n}+1,X_{2n}) \\
        &+1 \left\lbrace X_{2n} \neq 0, W_n= \vee \right\rbrace (X_{2n}-1,X_{2n}) \\
        &+1 \left\lbrace X_{2n} \neq 0, W_n= -, B_{n,X_{2n}}=1 \right\rbrace (X_{2n}+1,X_{2n}+2) \\
        &+1 \left\lbrace X_{2n} \neq 0, W_n= -, B_{n,X_{2n}}=0 \right\rbrace (X_{2n}-1,X_{2n}-2) \\
        &+1 \left\lbrace X_{2n} = 0, W_n \neq \wedge \right\rbrace (X_{2n}+1,X_{2n}+2)
    \end{align*}
    then the computation of the probability transitions for $(X_{2n+1},X_{2n+2})_{n \in \N}$ shows that indeed $X$ has the law of a symmetric random walk conditioned to stay non-negative starting from zero. Let us take an independent identically distributed family of random variables $(\tilde{B}_{n})_{n \in \N}$ of parameter $1/2$, and define inductively the Markov process $(S_{n})_{n \in \N}$ by $S_0=0$
    \begin{align*}
        (S_{2n+1},S_{2n+2}):=&1 \left\lbrace W_n= \wedge \right\rbrace (S_{2n}+1,S_{2n}) \\
        &+1 \left\lbrace W_n= \vee \right\rbrace (S_{2n}-1,S_{2n}) \\
        &+1 \left\lbrace W_n= -, \tilde{B}_n=1 \right\rbrace (S_{2n}+1,S_{2n}+2) \\
        &+1 \left\lbrace W_n= -, \tilde{B}_n=0 \right\rbrace (S_{2n}-1,S_{2n}-2) 
    \end{align*}
    then the computation probabilty transitions for $(S_{2n},S_{2n+1})_{n \in \N}$ shows that $S$ has the law of a simple symetric random walk (starting with a $+1$ step). Now, $S$ and $X$ as coupled via $W$ satisfy
    \begin{align*}
        \abs{ 1 \left\lbrace \Delta X_{2n+1} \neq 0 \right\rbrace - 1 \left\lbrace \Delta S_{2n+1} \neq 0 \right\rbrace }
        &\leq 1 \left\lbrace X_{2n}=0  \right\rbrace
    \end{align*}
    which proves \eqref{domination stochastique marche conditionnée}.
\end{proof}

\subsection{Moment estimate on the increments}

In \cite[Lemma 2.2]{lampertinewclass} it is proved that for all $k \in \N$, there exists a constant $a_k >0$ such that for all $n \in \N$
\begin{equation} \label{lamperti}
    \pi \left[ (X_n)^{2k} \right] \leq a_k n^k
\end{equation}
We use this to bound the increments in the following way.

\begin{lemma} \label{lemma moment estimate on the increments SSRW conditioned} For all $k \in \N$ there exists a constant $b_k >0$ such that
    \begin{equation} \label{bound from lamperti}
        \forall n,m \in \N \qquad \pi \left[ (X_n-X_m)^{2k} \right] \leq b_k \abs{n-m}^k
    \end{equation}
\end{lemma}

\begin{proof}
    Let us write $X \preccurlyeq Y$ to say that the random variable $X$ is stochastically dominated by $Y$. Without loss of generality, assume that $n \geq m$. First, from the inequality on the probability transitions $p(k,k+1) \geq p(k,k-1)$ for all $k \in \N$, we deduce that
    \begin{equation} \label{stochastic domination 1}
        (X_n-X_m)_- \preccurlyeq (X_n-X_m)_+.
    \end{equation}
    where $(\cdot)_+$ and $(\cdot)_-$ denote respectively the positive and negative parts.  Second, from the inequality on the probability transitions $p(k,k+1) \leq p(j,j+1)$ whenever $k \geq j$, we deduce that
    \begin{equation} \label{stochastic domination 2}
        (X_n-X_m)_+ \preccurlyeq (X_{n-m})_+.
    \end{equation}
    Consequently, using \eqref{stochastic domination 1} and \eqref{stochastic domination 2}
    \begin{align*}
        \pi \left[ (X_n-X_m)^{2k} \right] &= \pi \left[ ((X_n-X_m)_+)^{2k} \right] + \pi \left[ ((X_n-X_m)_-)^{2k} \right] \\
        &\leq 2 \pi \left[ ((X_n-X_m)_+)^{2k} \right] \\
        &\leq 2 \pi  \left[ ((X_{n-m})_+)^{2k} \right] \\
        &\leq 2 a_k (n-m)^k
    \end{align*}
    where we used \eqref{lamperti} in the last line.
\end{proof}

\section{Tightness of \texorpdfstring{$(h^\epsilon)_{\epsilon \in (0,1]}$}{h}}

\noindent In this section we fix $\rho>0$ and focus on the discrete interfaces, that is the collection $(h^\epsilon)_{\epsilon \in (0,1]}$ of $D([0,T], \mathcal{C}_\rho)$-valued random variables, where the space
\begin{equation*}
    \mathcal{C}_\rho:= \left\lbrace f \in C([0,\infty)) \:: \:  f(0)=0, \; \displaystyle{\sup_{x \in [0,\infty)}} \abs{f(x)} e^{-\rho x} =:\norme{f}_{\mathcal{C}_\rho} < \infty \right\rbrace,
\end{equation*}
\noindent is endowed with the topology induced by $\norme{\cdot}_{\mathcal{C}_\rho}$. The goal in this section is to prove the following result
\begin{theorem} \label{tightness of h}
    The collection $(h^\epsilon)_{\epsilon \in (0,1]}$ is tight in $D([0,\infty), \mathcal{C}_\rho)$ and any limit point belongs to $C([0,\infty), \mathcal{C}_\rho)$.
\end{theorem}
\noindent To do so, let us write $\bar{h}$ for the piecewise linear interpolation in time of $h$, that is 
$$\bar{h}_t:= (1- t + \lfloor t \rfloor) h_{\lfloor t \rfloor} + (t - \lfloor t \rfloor) h_{\lceil t \rceil} \qquad \forall t \geq 0. $$
Let us also write $\bar{h}^\epsilon$ for the rescaling of $\bar{h}$, that is
$$\bar{h}^\epsilon_t(x):=\bar{h}_{\epsilon^{-2}t}(\epsilon^{-1}x) \qquad \forall x \in \epsilon \N, \quad \forall t \geq 0.$$
Finally let us write $\bar{g}^\epsilon$ (resp. $g^{\epsilon}$) for the multiplication of $\bar{h}^\epsilon$ (resp. $h^\epsilon$) by an exponential factor, that is
\begin{align*}
    &\bar{g}^\epsilon_t(x):= e^{-\rho x} \bar{h}^\epsilon_t(x) \qquad \forall x \in \epsilon \N, \quad \forall t \geq 0 \\
    &g^\epsilon_t(x):= e^{-\rho x} h^\epsilon_t(x) \qquad \forall x \in \epsilon \N, \quad \forall t \geq 0.
\end{align*} 
It suffices to prove
\begin{proposition}
    The collection $(g^\epsilon)_{\epsilon \in (0,1]}$ is tight in $D([0,\infty), C([0,\infty)))$ and any limit point belongs to $C([0,\infty), C([0,\infty)))$.
\end{proposition}
\noindent In what follows, we will extend the functions $h_t^\epsilon, \bar{h}_t^\epsilon, g_t^\epsilon, \bar{g}_t^\epsilon$ to functions on $\R$, setting their value to zero for $x \in (-\infty, 0)$.

\subsection{Moment estimate for the \texorpdfstring{$\mathcal{W}^{s_1, r}$}{W}-norm of the time increments} \label{subsection w s1 r}

For any $s_1 > 0$ and $r \geq 1$, let us introduce the Sobolev-Slobodeckij space
\begin{equation}
    \mathcal{W}^{s_1,r}:= \left\lbrace f \in L^r(\R) \:: \: \norme{f}_{L^r(\R)}^{r} + \Int{\R^2}{} \f{\abs{f(x)-f(y)}^{r}}{\abs{x-y}^{s_1 r+1}} \: dx \, dy =: \norme{f}^{r}_{\mathcal{W}^{s_1,r}} < \infty \right\rbrace.
\end{equation}
The aim of this paragraph is to prove the following statement
\begin{lemma} \label{estimation positive sobolev} For every $s_1 \in (0,1/2)$, every $r, p >1$ and $s,t \in [0,T]$ we have
    \begin{equation} \label{equation estimation positive sobolev}
        \displaystyle{\sup_{\epsilon \in (0,1]} \:  \mathbb{E} \left[ \norme{\bar{g}^\epsilon_t -\bar{g}^\epsilon_s}_{\mathcal{W}^{s_1,r} } ^p \right]^{\f{1}{p}} } < \infty.
    \end{equation}
\end{lemma}

\begin{proof}
	Let us first observe that, using Minkowski inequality on the $L^{2k}$-norm and the concavity of $x\mapsto x^{1/2}$, the bounds \eqref{lamperti} and \eqref{bound from lamperti} can be lifted at the level of the piecewise affine process as follows: for all $t\ge 0$, $x,y \in \R_+$ and all $k\in \N$
	\begin{equation}\label{Eq:akbk}
		\E[|h_t(x)|^{2k}] \le a_k x^k\;,\quad \E[|h_t(x)-h_t(y)|^{2k}] \le b_k |x-y|^k\;.
	\end{equation}
    We now prove the bound of the statement. Without loss of generality, we can assume that $p>r$. By the triangle inequality
    \begin{equation*}
        \mathbb{E} \left[ \norme{\bar{g}^\epsilon_t -\bar{g}^\epsilon_s}_{\mathcal{W}^{s_1,r} } ^p \right]^{1/p}
        \leq \mathbb{E} \left[ \norme{\bar{g}^\epsilon_t}_{\mathcal{W}^{s_1,r} } ^p \right]^{1/p}
        + \mathbb{E} \left[ \norme{\bar{g}^\epsilon_s}_{\mathcal{W}^{s_1,r} } ^p \right]^{1/p}.
    \end{equation*}
    Then, as $\bar{g}^\epsilon$ corresponds to the linear time interpolation of $g^\epsilon$,
    \begin{align*}
         \mathbb{E} \left[ \norme{\bar{g}^\epsilon_t}_{\mathcal{W}^{s_1,r} } ^p \right]^{1/p}
         &\leq \mathbb{E} \left[ \norme{g^\epsilon_{\lfloor t \epsilon^{-2} \rfloor \epsilon^2}} _{\mathcal{W}^{s_1,r} } ^p \right]^{1/p} +
        \E \left[ \norme{g^\epsilon_{\lceil t \epsilon^{-2} \rceil \epsilon^2}} _{\mathcal{W}^{s_1,r} } ^p \right]^{1/p} \\
        &\leq 2 \E \left[ \norme{g^\epsilon_0} _{\mathcal{W}^{s_1,r} } ^p \right]^{1/p}
    \end{align*}
    using in the last line the fact that the process starts from stationarity. We now estimate the last term. 
    First, we have by H\"older's inequality and \eqref{Eq:akbk}
    \begin{align*}
        \E \left[ \left( \Int{[0,\infty)}{} \abs{g_0^\epsilon (x)}^{r} \: dx\right)^p \right] &\leq C_1^{p-1} \E \left[ \Int{[0,\infty)}{} {\abs{h_0^\epsilon(x)}}^{rp} e^{-rp\rho x} \: dx \right] \\
        &\leq C_1^{p-1} a_{rp/2} \Int{[0,\infty)}{} x^{rp/2} e^{-rp\rho x} \: dx \\
        &< \infty.
    \end{align*}
    with $C_1:=\Int{[0,\infty)}{} \left( e^{-r \rho x/2} \right)^{\f{p}{p-1}} \: dx$.
    Second, we have, for all $x, y \in [0,\infty)$ such that $\abs{x-y} \leq 1$
    \begin{align*}
        \abs{g_0^\epsilon(x)-g_0^\epsilon(y)} &=\abs{h_0^\epsilon(x)e^{-\rho x}-h_0^\epsilon(y)e^{-\rho y}} \\
        &\leq \abs{h_0^\epsilon(x)} \abs{e^{-\rho x}-e^{-\rho y}} + e^{-\rho y} \abs{h_0^\epsilon(x)-h_0^\epsilon(y)} \\
        &\leq \abs{h_0^\epsilon(x)} e^{-\rho (x-1)} \rho \abs{x-y} + e^{-\rho y} \abs{h_0^\epsilon(x)-h_0^\epsilon(y)} \\        
    \end{align*}
    Thus, by the triangle inequality
        \begin{align} \label{triangular inequality}
        \mathbb{E}& \left[ \left( \Int{x}{} \Int{\abs{y-x}\leq 1}{} \abs{g_0^\epsilon(x)-g_0^\epsilon(y)}^r \: \f{dx \, dy}{\abs{x-y}^{1+s_1 r}} \right)^{p/r}  \right]^{1/p} \nonumber \\
        &\leq \mathbb{E} \left[ \left( \Int{x}{} \rho^r \Int{\abs{y-x}\leq 1}{} \abs{h_0^\epsilon(x)}^{r} e^{- r \rho (x-1)} \abs{x-y}^{r}  \: \f{dx \, dy}{\abs{x-y}^{1+s_1 r}} \right)^{p/r}  \right]^{1/p} \nonumber \\
        &+ \mathbb{E} \left[ \left( \Int{x}{} \Int{\abs{y-x}\leq 1}{} e^{-r\rho y} \abs{h_0^\epsilon(x)-h_0^\epsilon(y)}^{r} \: \f{dx \, dy}{\abs{x-y}^{1+s_1 r}} \right)^{p/r}  \right]^{1/p}
    \end{align}
	We start by bounding the first term on the right hand side of \eqref{triangular inequality}. Note that
	$C_2 := \Int{x}{} \Int{\abs{x-y}<1}{} e^{- r \rho (x-1)} \abs{x-y}^r  \: \f{dx \, dy}{\abs{x-y}^{1+s_1 r}} < \infty$. Thus by Jensen's inequality with the convex function $x\mapsto x^{p/r}$ and \eqref{Eq:akbk}
    \begin{align*}
        & \mathbb{E} \left[ \left( \Int{x}{} \Int{\abs{y-x}\leq 1}{} \abs{h_0^\epsilon(x)}^{r} e^{- r \rho (x-1)} \abs{x-y}^r  \: \f{dx \, dy}{\abs{x-y}^{1+s_1 r}} \right)^{p/r}  \right]^{1/p} \\
        &\leq C_2^{\f{p-r}{rp}} \mathbb{E} \left[ \Int{x}{} \Int{\abs{y-x}\leq 1}{} \abs{h_0^\epsilon(x)}^{p} e^{- r \rho (x-1)} \abs{x-y}^r  \: \f{dx \, dy}{\abs{x-y}^{1+s_1 r}}   \right]^{1/p} \\
        &\leq C_2^{\f{p-r}{rp}} \left( \Int{x}{} a_{p/2} \abs{x}^{p/2} e^{- r \rho (x-1)/2} \: dx \Int{\abs{u}\leq 1}{} \abs{u}^{-1+r(1-s_1)} \: du \right)^{1/p} < \infty\;.
    \end{align*}
    Similarly the second term on the right hand side of \eqref{triangular inequality} can be bounded as follows
    \begin{align*}
        & \mathbb{E} \left[ \left( \Int{x}{} \Int{\abs{y-x}\leq 1}{} e^{-r\rho y} \abs{h_0^\epsilon(x)-h_0^\epsilon(y)}^r \: \f{dx \, dy}{\abs{x-y}^{1+s_1 r}} \right)^{p/r}  \right]^{1/p} \\
        &= \mathbb{E} \left[ \left( \Int{x}{} \Int{\abs{y-x}\leq 1}{} e^{-r\rho y} \frac{\abs{h_0^\epsilon(x)-h_0^\epsilon(y)}^r}{\abs{x-y}^{\frac{r}{2}}} \: \f{dx \, dy}{\abs{x-y}^{1+(s_1-\frac12) r}} \right)^{p/r}  \right]^{1/p} \\
        &\leq C_3^{\f{p-r}{rp}}\mathbb{E} \left[ \Int{x}{} \Int{\abs{y-x}\leq 1}{} e^{-r\rho y} \frac{\abs{h_0^\epsilon(x)-h_0^\epsilon(y)}^p}{\abs{x-y}^{\frac{p}{2}}} \: \f{dx \, dy}{\abs{x-y}^{1+(s_1-\frac12) r}}  \right]^{1/p}\\
        &\leq C_3^{\f{p-r}{rp}} \left( \Int{x}{} b_{p/2}  e^{- r \rho x} \: dx \Int{\abs{u}\leq 1}{} \abs{u}^{-1+r(\frac12 - s_1)}   \: du   \right)^{1/p} < \infty
    \end{align*}
    with $C_3:=\Int{x}{} \Int{\abs{x-y}<1}{} e^{- r \rho (x-1)} \abs{x-y}^{r/2}  \: \f{dx \, dy}{\abs{x-y}^{1+s_1 r}} < \infty$, and where we used \eqref{Eq:akbk}.
    Third, we have
    \begin{align*}
        \mathbb{E}& \left[ \left( \Int{x}{} \Int{\abs{y-x} > 1}{} \abs{g^\epsilon_0(x)-g_0^\epsilon(y)}^{r} \: \f{dxdy}{\abs{x-y}^{1+s_1 r}} \right)^{p/r}  \right]^{1/p} \\
        &\leq 2 \mathbb{E} \left[ \left( \Int{x}{} \Int{\abs{y-x} > 1}{} \abs{h^\epsilon_0(x)}^r e^{-r\rho x} \: \f{dxdy}{\abs{x-y}^{1+s_1 r}} \right)^{p/r}  \right]^{1/p} \\
        &\leq 2 C_4^{\f{p-r}{rp}} \mathbb{E} \left[ \Int{x}{} \Int{\abs{y-x} > 1}{} \abs{h^\epsilon_0(x)^p}e^{-p \rho x/2}   \: \f{dxdy}{\abs{x-y}^{1+s_1 r}}   \right]^{1/p} \\
        &\leq 2 C_4^{\f{p-r}{rp}} \Int{x}{} \Int{\abs{y-x} > 1}{} a_{p/2} \abs{x}^{p/2} e^{-p \rho x/2}  \: \f{dxdy}{\abs{x-y}^{1+s_1 r}}  < \infty
    \end{align*}
    with $C_4:= \Int{x}{}  \Int{\abs{x-y}\geq 1}{} (e^{-r \rho x/2})^{\f{p}{p-r}} \: \f{dxdy}{\abs{x-y}^{1+s_1 r}}$, and where we used Lemma \ref{lemma moment estimate on the increments SSRW conditioned} in the fourth line. The last three points conclude the proof of \eqref{equation estimation positive sobolev}.

\end{proof}

\subsection{Moment estimate for the \texorpdfstring{$\mathcal{H}^{-s_0}$}{H}-norm of the time increments} \label{subsection w -s_0}
For any $s_0 \geq 0$, let us introduce the Sobolev space of distributions
\begin{equation}
    \mathcal{H}^{- s_0}:=\left\lbrace f \in \mathcal{S}'(\R) \:: \: \Int{\R}{} \left( 1+\abs{\zeta}^2 \right)^{- s_0} \abs{\hat{f}(\zeta)}^2 \: d\zeta =: \norme{f}_{\mathcal{H}^{-s_0}}^2 < \infty \right\rbrace,
\end{equation}
where $\hat{f}$ denotes the Fourier transform of $f$. The aim of this paragraph is to prove the following result.
\begin{proposition} \label{estimation negative sobolev}
    For any $s_0 >1/2$ and any integer $p\geq 1$ there exists $c>0$ such that for every $0 \leq s \leq t \leq T$
        \begin{equation}
        \displaystyle{\sup_{\epsilon \in (0,1]} \; \mathbb{E} \left[ \norme{\bar{g}^\epsilon_t -\bar{g}^\epsilon_s}_{\mathcal{H}^{-s_0}}^{2p} \right]^{\f{1}{2p}} \leq c (t-s)^{3/8}}.
    \end{equation}
\end{proposition}

\noindent To prove the above statement, we need the following estimate on the time increments of the Fourier transform of $g^\epsilon$.

\begin{lemma} \label{Holder moment estimate sobolev} For any $T>0$, any integer $m \geq 1$, there exists a constant $c_{m,T,\rho}>0$ such that for any $0 \leq s \leq t \leq T$ any $\epsilon \in (0,1]$, and any $\zeta \in \R$
    \begin{equation}
        \norme{\hat{g}^\epsilon_t(\zeta) -\hat{g}^\epsilon_s(\zeta)}_{L^m(\Omega)} \leq c_{m,T,\rho} \left( (t-s)^{1/2} +\epsilon^{3/4}  \right).
    \end{equation}
\end{lemma}

\begin{proof}
    First, let us use Lyons-Zheng's decomposition~\cite{lyonszheng} to reduce the above estimate to a control on the moments of some martingale. For $\zeta \in \R$, let us define $e_\zeta:x \mapsto e^{-i\zeta x}$ . We consider the Markov process $(h^\epsilon_t)_{t \geq 0}$ and denote $L^\epsilon$ its generator. Then, applying Dynkin's formula to the Markov process $h^\epsilon$ and the function $f_\zeta(\cdot):=\left\langle \cdot , e_\zeta \right\rangle_{L^2([0,\infty),e^{-\rho x}dx)}$, we obtain
    \begin{equation} \label{forward Dynkin}
        f_\zeta(h^\epsilon_t) = f_\zeta(h^\epsilon_s) + \Int{s}{t} L^\epsilon f_\zeta (h^\epsilon_r) \: dr + \hat{M}_{s,t}(\zeta)
    \end{equation}
    where the process $\hat{M}_{s,\cdot}(\zeta)$ is a martingale.
    Additionally, writing Dynkin's formula for the backward process we obtain
    \begin{align*}
    f_\zeta(h^\epsilon_{T-(T-s)})& = f_\zeta(h^\epsilon_{T-(T-t)}) + \Int{T-t}{T-s} L^\epsilon 
    f_\zeta(h^\epsilon_{T-r}) \: dr + \hat{N}_{s,t}(\zeta) 
    \end{align*}
    where the process $\hat{N}_{s,\cdot}(\zeta)$ is a backward martingale. Note that to obtain the last equality, we used the fact that the dynamic is reversible with respect to $\pi$ and that we start from stationarity, which implies that the generator of the backward process is identical to the one of the forward process. Now last equation rewrites
    \begin{equation} \label{backward Dynkin}
    f_\zeta(h^\epsilon_{s}) = f_\zeta(h^\epsilon_{t}) + \Int{s}{t} L^\epsilon f_\zeta(h^\epsilon_r) \: dr + \hat{N}_{s,t}(\zeta) .
    \end{equation}
    Subtracting the forward and backward equations \eqref{forward Dynkin} and \eqref{backward Dynkin}, we obtain
    \begin{align*}
        \hat{g}^\epsilon_t(\zeta)-\hat{g}^\epsilon_s(\zeta) 
        &= f_\zeta(h^\epsilon_t)-f_\zeta(h^\epsilon_s) 
        = \f{1}{2}\left[ \hat{M}_{s,t}(\zeta) - \hat{N}_{s,t}(\zeta)  \right].
    \end{align*}
    Without loss of generality we can focus on the forward martingale. Let us start by giving a more explicit formula for the martingale term $\hat{M}_{s,t}(\zeta)$. Using the expression given by Lemma \ref{lemma expression fourrier coefficients} for the Fourier transform of $g^\epsilon_t$, and the rescaled semimartingale equation, for $\zeta \in \R$ we have
    \begin{multline*}
    \hat{g}^\epsilon_t(\zeta)-\hat{g}^\epsilon_s(\zeta) =c_{\zeta,\epsilon} \sum_{x \in \epsilon \N^*}{} e^{-i\zeta x} e^{- \rho x} \left[ h^\epsilon_t(x)-h^\epsilon_s(x) \right] \\
    =c_{\zeta,\epsilon} \sum_{x \in \epsilon \N^*}{} e^{-i\zeta x}e^{- \rho x} \f{1}{\epsilon^2} \Int{s}{t} \Delta^\epsilon h^\epsilon_r(x) 1 \left\lbrace h^\epsilon_r(x) +\Delta^\epsilon h^\epsilon_r(x) \geq 0 \right\rbrace \: dr \\
    +c_{\zeta,\epsilon} \sum_{x \in \epsilon \N^*}{} e^{-i\zeta x} e^{- \rho x}  \Int{s}{t} \Delta^\epsilon h^\epsilon_r(x) 1 \left\lbrace h^\epsilon_r(x) +\Delta^\epsilon h^\epsilon_r(x) \geq 0 \right\rbrace \: dM^\epsilon_r(x).
    \end{multline*}
    Comparing with the forward Dynkin's formula, and using uniqueness of the decomposition for a semimartingale, we obtain
    \begin{equation*}
        \hat{M}_{s,t}(\zeta)=c_{\zeta,\epsilon} \sum_{x \in \epsilon \N^*}{} e^{-i\zeta x} e^{- \rho x}  \Int{s}{t}  \Delta^\epsilon h^\epsilon_r(x) 1 \left\lbrace h^\epsilon_r(x) +\Delta^\epsilon h^\epsilon_r(x) \geq 0 \right\rbrace \: dM^\epsilon_r(x).
    \end{equation*}
    By independence of the martingales $(M^\epsilon_\cdot(x ))_{x \in \epsilon \N^*}$ the bracket of $\hat{M}_{s,\cdot}(\zeta)$ writes as
    \begin{equation*}
        \langle\!\langle \hat{M}_{s,\cdot} (\zeta) \rangle\!\rangle_t = c_{\zeta,\epsilon}^2 \sum_{x \in \epsilon \N^*}{} e^{-2i\zeta x} e^{-2 \rho x} \Int{s}{t} \Delta^\epsilon h^\epsilon_r(x)^2 1 \left\lbrace h^\epsilon_r(x) +\Delta^\epsilon h^\epsilon_r(x)\geq 0 \right\rbrace \: \f{1}{\epsilon^2}dr
    \end{equation*}
    which, recalling \eqref{def and bound for coef c}, is bounded as follows
\begin{align} \label{bound bracket}
    \abs{\langle\!\langle \hat{M}_{s,\cdot} (\zeta) \rangle\!\rangle_t }
    &\lesssim \epsilon^2 \f{1}{ \epsilon} (2\sqrt{\epsilon})^2 (t-s) \f{1}{\epsilon^2}  \nonumber \\
    &\lesssim (t-s)
\end{align}
where the constant involved in $\lesssim$ depends only on $\rho$, in particular it is uniform in $\epsilon \in (0,1]$.
Then, turning to the quadratic variation term, we have
\begin{equation*}
    \left[ \hat{M}_{s,\cdot} (\zeta) \right]_t = 
    c_{\zeta,\epsilon}^2 \sum_{x \in \epsilon \N^*}{} e^{-2i\zeta x} e^{-2 \rho x}  \sum_{s \leq \tau \leq t}{} \Delta^\epsilon h^\epsilon_\tau (x)^2 1 \left\lbrace h^\epsilon_\tau  (x) +\Delta^\epsilon h^\epsilon_\tau (x) \geq 0 \right\rbrace \left( M^\epsilon_\tau(x) -M^\epsilon_{\tau^-}(x) \right)^2.
\end{equation*}
Setting $\hat{D}_{s,t}(\zeta):=\left[ \hat{M}_{s,\cdot} (\zeta) \right]_t - \langle\!\langle \hat{M}_{s,\cdot} (\zeta) \rangle\!\rangle_t$, then $\hat{D}_{s,\cdot}(\zeta)$ is a martingale and
\begin{equation*}
    \left[ \hat{D}_{s,\cdot}(\zeta) \right]_t = c_{\zeta,\epsilon}^4 \sum_{x \in \epsilon \N^*}{} e^{-4i\zeta x} e^{-4 \rho x}  \sum_{s \leq \tau \leq t}{}
    \Delta^\epsilon h^\epsilon_\tau (x)^4 1 \left\lbrace h^\epsilon_\tau +\Delta^\epsilon h^\epsilon_\tau (x) \geq 0 \right\rbrace \left( M^\epsilon_\tau(x) -M^\epsilon_{\tau^-}(x) \right)^4.
\end{equation*}
We obtain the following bound
\begin{align} \label{bound quadratic variation}
    \left\|  \left[ \hat{D}_{s,\cdot}(\zeta) \right]_t \right\|_{L^m(\Omega)} &\leq c_{\zeta,\epsilon}^4 \f{1}{\epsilon} (2\sqrt{\epsilon})^4 \norme{\mathcal{P}(\epsilon^{-2}(t-s))}_{L^m(\Omega)} \nonumber \\ \nonumber
    &\lesssim \epsilon^{4 -1 +2} \left( \epsilon^{-2} (t-s) + (\epsilon^{-2} (t-s))^{1/m} \right) \\ \nonumber
    &\lesssim \epsilon^3 \left( (t-s) + T^{1/m} \right) \\
    &\lesssim \epsilon^3
\end{align}
where we used the abuse of notation $\mathcal{P}(\lambda)$ to denote a Poisson random variable of parameter $\lambda$, and where the constant involved in $\lesssim$ depends only on $m$, $\rho$ and $T$, in particular it is uniform in $\epsilon \in (0,1]$. Now applying twice the Burkholder-Davis-Gundy formula yields the following (general) inequality
\begin{equation}
    \norme{\hat{M}_{s,t}(\zeta)}_{L^{m}(\Omega)} \leq c_{BDG}(m) \left( \norme{\langle\!\langle \hat{M}_{s,t}(\zeta) \rangle\!\rangle_t}_{L^{\f{m}{2}}(\Omega)}^{\f{1}{2}} + c_{BDG} \left(\f{m}{2} \right)^{\f{1}{2}} \norme{\hat{D}_{s,t}(\zeta)}_{L^{\f{m}{4}}(\Omega)}^{\f{1}{4}} \right).
\end{equation}
Combining this with estimates \eqref{bound bracket} and \eqref{bound quadratic variation}, we obtain
\begin{equation}
    \norme{\hat{M}_{s,t}(\zeta)}_{L^{m}(\Omega)} \lesssim \left( (t-s)^{\f{1}{2}} + \epsilon^{\f{3}{4}} \right)
\end{equation}
with the constant involved in $\lesssim$ depending only on $m$, $\rho$ and $T$, in particular it is uniform in $\epsilon \in (0,1]$. This concludes the proof.
\end{proof}

\begin{lemma} (Linearization in time) \label{linearization in time} For every integer $m\geq 1$ there exists $c(m)>0$ such that for all $0 \leq s \leq t \leq T$
    \begin{equation*}
        \displaystyle{\sup_{\epsilon \in (0,1], \zeta \in \R}} \norme{\hat{\bar{g}}^\epsilon_t(\zeta) -\hat{\bar{g}}^\epsilon_s(\zeta)}_{L^m(\Omega)} \leq c(m) (t-s)^{3/8}
    \end{equation*}
\end{lemma}

\begin{proof} First, consider the case where there exists $p \in \epsilon^2 \N$ such that $s,t \in [p,p+\epsilon^2]$. Then 
    \begin{align*}
        \hat{\bar{g}}^\epsilon_t(\zeta)-\hat{\bar{g}}^\epsilon_s(\zeta)&= c_{\zeta,\epsilon} \sum_{x \in \epsilon \N}{} e^{-i\zeta x} \left[ \bar{g}_t^\epsilon (x) - \bar{g}_s^\epsilon (x) \right] \\
        &= c_{\zeta,\epsilon} \sum_{x \in \epsilon \N}{} e^{-i\zeta x} \f{(t-s)}{\epsilon^2} \left[ g_{p+\epsilon^2}^\epsilon (x) - g_{p}^\epsilon (x) \right] \\
        &= \f{(t-s)}{\epsilon^2} \left[ \hat{g}_{p+\epsilon^2}^\epsilon (\zeta) - \hat{g}_{p}^\epsilon (\zeta) \right]
    \end{align*}
    Thanks to Lemma \ref{Holder moment estimate sobolev} we obtain
    \begin{align*}
        \norme{\hat{\bar{g}}^\epsilon_{t}(\zeta)-\hat{\bar{g}}^\epsilon_s(\zeta)}_{L^m(\Omega)} &\lesssim \f{(t-s)}{\epsilon^2} (\epsilon +\epsilon^{3/4}) \\
        &\lesssim (t-s)\epsilon^{-5/4} \\
        &\lesssim (t-s)^{3/8}
    \end{align*}
    the last line coming from the fact that $0 \leq t-s \leq \epsilon^2$.
    Second, consider the case where $s,t$ do not both belong to a same interval $[p,p+\epsilon^2]$ for some $p \in \epsilon \N$. Then let $p_t:= \lfloor t\epsilon^{-2} \rfloor \epsilon^2 $ and $p_s:= \lceil s\epsilon^{-2} \rceil \epsilon^2$. If $p_t>p_s$ then
    \begin{align*}
        \norme{\hat{\bar{g}}^\epsilon_t(\zeta)-\hat{\bar{g}}^\epsilon_s(\zeta)}_{L^m(\Omega)}&\leq \norme{\hat{\bar{g}}^\epsilon_t(\zeta)-\hat{\bar{g}}^\epsilon_{p_t}(\zeta)}_{L^m(\Omega)} + \norme{\hat{\bar{g}}^\epsilon_{p_t}(\zeta)-\hat{\bar{g}}^\epsilon_{p_s}(\zeta)}_{L^m(\Omega)} + \norme{\hat{\bar{g}}^\epsilon_{p_s}(\zeta)-\hat{\bar{g}}^\epsilon_s(\zeta)}_{L^m(\Omega)} \\
        &\lesssim c (t-p_t)^{3/8} + c \left[ (p_t-p_s)^{1/2} + \epsilon^{3/4} \right] + c (p_s -s)^{3/8} \\
        &\lesssim c (t-p_t)^{3/8} +  c \left[ (p_t-p_s)^{1/2} + (t-s)^{3/8} \right]+ c (p_s -s)^{3/8} \\
        &\lesssim c (t-s)^{3/8},
    \end{align*}
    the fourth line coming from the fact that $t-s \geq \epsilon^2$. If $p_t=p_s$, the same computation applies except that $\norme{\hat{\bar{g}}^\epsilon_{p_t}(\zeta)-\hat{\bar{g}}^\epsilon_{p_s}(\zeta)}_{L^m(\Omega)}$ vanishes.
\end{proof}

\begin{proof}[Proof of Proposition  \ref{estimation negative sobolev}]
   We have
    \begin{align*}
        \mathbb{E} \left[ \norme{\bar{g}^\epsilon_t-\bar{g}^\epsilon_s}^{2p}_{\mathcal{H}^{-s_0}} \right] 
        &= \mathbb{E} \left[ \Int{\R^p}{} \prod_{j=1}^{p} \left( 1+\abs{\zeta_j}^2 \right)^{-s_0} \prod_{j=1}^{p} \abs{\hat{\bar{g}}^\epsilon_t(\zeta_j) -\hat{\bar{g}}^\epsilon_s(\zeta_j)}^2 \: d\zeta_1 \cdots d\zeta_p \right] \\
        &=  \Int{\R^p}{} \prod_{j=1}^{p} \left( 1+\abs{\zeta_j}^2 \right)^{-s_0} \mathbb{E} \left[\prod_{j=1}^{p} \abs{\hat{\bar{g}}^\epsilon_t(\zeta_j) -\hat{\bar{g}}^\epsilon_s(\zeta_j)}^2 \right] \: d\zeta_1 \cdots d\zeta_p \ \\
        &\leq \Int{\R^p}{} \prod_{j=1}^{p} \left( 1+\abs{\zeta_j}^2 \right)^{-s_0} \prod_{j=1}^{p} \mathbb{E} \left[  \abs{\hat{\bar{g}}^\epsilon_t(\zeta_j) -\hat{\bar{g}}^\epsilon_s(\zeta_j)}^{2^{j+1}}  \right]^{\f{1}{2^j}} \: d\zeta_1 \cdots d\zeta_p \ \\
        &\leq \Int{\R^p}{} \prod_{j=1}^{p} \left( 1+\abs{\zeta_j}^2 \right)^{-s_0} \prod_{j=1}^{p} \left( c(2^{j+1}) (t-s)^{3/8} \right)^2 \: d\zeta_1 \cdots d\zeta_p \ \\
        &\leq (t-s)^{6p/8}  \prod_{j=1}^{p} c(2^{j+1})^2 \left[ \Int{- \infty}{\infty} \left( 1+ \abs{\zeta}^2 \right)^{-s_0} \: d\zeta \right]^{p} \\
        &\leq c (t-s)^{3p/4}
    \end{align*}
    where we used Cauchy-Schwarz inequality to obtain the third line and Lemma \ref{linearization in time} to obtain the fourth line.
\end{proof}

\subsection{Moment estimate for the \texorpdfstring{$C^b$}{Cb}-norm of the time increments}  For any $b>0$ let us introduce the Hölder space
\begin{equation}
    C^b:=\left\lbrace f \in L^\infty(\R) \:: \: \norme{f}_{L^\infty} + \displaystyle{\sup_{x \neq y}} \: \f{\abs{f(x)-f(y)}}{\abs{x-y}^{b}}=: \norme{f}_{C^b} < \infty \right\rbrace.
\end{equation}
The aim of this paragraph is to use an interpolation and embedding argument in order to deduce from the results from the two preceding paragraphs, a moment estimate on the Hölder norm of the time-increments. More precisely we prove the following
\begin{lemma} \label{lemme Cb}
    For any $b \in (0, \f{1}{2})$, there exists $\kappa >0$ such that for all $p >1$ there exists a constant $c>0$ such that
    \begin{equation} \label{estimate Cb norm}
        \forall s,t \in [0,T] \qquad \displaystyle{\sup_{\epsilon  \in (0,1]}} \: \mathbb{E} \left[ \norme{\bar{g}^{\epsilon}_t - \bar{g}^{\epsilon}_s }_{C^b}^{2p} \right] \leq c \abs{t-s}^{\kappa p}.
    \end{equation}
\end{lemma}

\begin{proof}
By interpolation between Sobolev spaces \cite[p.182 Section 2.4.1 Theorem c)]{triebel1995interpolation}, given $s_0, s_1 \in \R$ and $r_1 \in (1, \infty)$, for all $\theta \in [0,1]$, there exists a constant $c_{\textrm{Interpo}} >0 $ such that
\begin{equation} \label{interpolation triebel}
    \norme{f}_{\mathcal{W}^{\delta,r}} \leq c_{\textrm{Interpo}} \norme{f}_{\mathcal{W}^{s_1,r_1}}^\theta \norme{f}_{\mathcal{H}^{-s_0}}^{1-\theta} \qquad \forall f \in \mathcal{W}^{s_1,r_1} \cap \mathcal{H}^{-s_0}
\end{equation}
where 
\begin{equation} \label{barycentre interpo}
    \left\lbrace
    \begin{array}{rl}
         \delta &:=(1-\theta) (- s_0) +\theta s_1  \\
         \f{1}{r}&:=\f{1-\theta}{2}+\f{\theta}{r_1} 
    \end{array}
    \right.
\end{equation} Our aim is now to fix all the parameters in such a way that $\mathcal{W}^{\delta,r}$ is continuously embedded in $C^b$, and such that the choice of $s_0$, $s_1$ and $r_1$ allows us to apply the results from Subsections \ref{subsection w s1 r} and \ref{subsection w -s_0}. 
Let $b \in (0,1/2)$. Fix $s_0> 1/2$ and $s_1 \in (b,1/2)$. Then let us take $\delta \in (b, s_1)$ close enough to $s_1$ so that
\begin{equation} \label{choice of delta}
    \delta - \f{s_1-\delta}{2(s_1+s_0)} >b.
\end{equation}
From \eqref{choice of delta}, we can take $r_1>1$ large enough so that 
\begin{equation} \label{allowing embedding}
    \delta - \f{s_1-\delta}{2(s_1+s_0)} - \f{1}{r_1} >b
\end{equation}
We can then set 
\begin{equation*}
    \theta := \f{s_0 + \delta}{s_0 + s_1} \in (0,1) \qquad \text{and} \qquad r:=\left( \f{1-\theta}{2}+\f{\theta}{r_1} \right)^{-1}
\end{equation*}
to obtain \eqref{barycentre interpo} with our choice of parameters.
So by interpolation we obtain \eqref{interpolation triebel} for our choice of parameters $s_0, \delta, s_1, r$ and $r_1$. Using Hölder's inequality, we obtain
\begin{equation*}
    \mathbb{E} \left[ \norme{ \bar{g}^\epsilon_t -  \bar{g}^\epsilon_s }_{\mathcal{W}^{\delta,r}}^p \right] 
    \leq c_{\textrm{Interpo}}^p
    \mathbb{E} \left[ \norme{ \bar{g}^\epsilon_t -  \bar{g}^\epsilon_s }_{\mathcal{W}^{s_1,r_1}}^p \right]^{\theta}
    \mathbb{E} \left[ \norme{ \bar{g}^\epsilon_t -  \bar{g}^\epsilon_s }_{\mathcal{H}^{-s_0}}^p \right]^{1-\theta}
\end{equation*}
 Additionally \eqref{allowing embedding} ensures that
\begin{equation}
    \delta - \f{1}{r} = \delta - \f{s_1-\delta}{2(s_1+s_0)} - \f{\theta}{r_1} >b.
\end{equation} 
Consequently, we have the continuous embedding $\mathcal{W}^{\delta,r} \hookrightarrow C^b$. Now together with Lemma \ref{estimation positive sobolev} and Proposition \ref{estimation negative sobolev}, this concludes the proof.
\end{proof}

\subsection{Estimation of the interpolation error}
The aim of this paragraph is to control the error from the linear time interpolation. More precisely we prove the following
\begin{lemma} \label{estimation of the linearization error}
    For all $p \geq 1$, we have
    \begin{equation}
        \displaystyle{\lim_{\epsilon \to 0} \: \mathbb{E} \left[ \displaystyle{\sup_{t \in [0,T]}} \norme{\bar{g}^\epsilon_t - g^\epsilon_t}_{\infty}^p \right]} =0
    \end{equation}
\end{lemma}

\begin{proof}
    Take $p\geq 1$. For $i,k \in \N$, let us denote $B_{k,i}:=[i\epsilon^2, (i+1)\epsilon^2]\times [k\epsilon,(k+1)\epsilon] $. We have
    \begin{align} \label{découpage en sup sur les boites}
        \mathbb{E} \left[ \displaystyle{\sup_{\substack{t \in [0,T] \\ x \in [0,\infty)}}} \abs{\bar{g}^\epsilon_t(x)-g^{\epsilon}_t(x)}^p \right]
        &=\mathbb{E} \left[ \displaystyle{\sup_{\substack{t \in [0,T] \\ x \in [0,\infty)}}} e^{-\rho xp} \abs{\bar{h}^\epsilon_t(x)-h^{\epsilon}_t(x)}^p \right] \nonumber \\
        &\leq \sum_{i=0}^{\lfloor \epsilon^{-2} T \rfloor} \sum_{k=0}^{\infty} \mathbb{E} \left[ \displaystyle{\sup_{\substack{(t,x) \in B_{k,i}}}} e^{-\rho xp} \abs{\bar{h}^\epsilon_t(x)-h^{\epsilon}_t(x)}^p \right] \nonumber \\
        &\leq \sum_{i=0}^{\lfloor \epsilon^{-2} T \rfloor} \sum_{k=0}^{\infty} e^{-\rho kp}  \mathbb{E} \left[ \displaystyle{\sup_{\substack{(t,x) \in B_{k,i} }}}\abs{\bar{h}^\epsilon_t(x)-h^{\epsilon}_t(x)}^p \right]
    \end{align}
    Let us bound the expectation term on the right hand side. For any $(t,x) \in B_{k,i}$
    \begin{align*} 
    \abs{\bar{h}^\epsilon_t(x)-h^\epsilon_t(x)} &\leq \abs{\bar{h}^\epsilon_t(x)-\bar{h}^\epsilon_t(k \epsilon)} 
    + \abs{\bar{h}^\epsilon_t(k \epsilon)-h^\epsilon_t(k \epsilon)}
    + \abs{h^\epsilon_t(k \epsilon)-h^\epsilon_t(x)} \\
    &\leq 2\sqrt{\epsilon} + \abs{h_{(i+1)\epsilon^2}^\epsilon(k\epsilon) -h_{i\epsilon^2}^\epsilon(k\epsilon)} + 2 \sqrt{\epsilon} \\
    &\leq 4\sqrt{\epsilon} + \sqrt{\epsilon} \left( N_{i+1}(k)-N_i(k) \right)
    \end{align*}
    Since $N_{i+1}(k)-N_i(k) \sim \mathcal{P}(1)$, we deduce that
    \begin{equation} \label{estimate sup on space time box}
        \mathbb{E} \left[ \displaystyle{\sup_{\substack{(t,x) \in B_{k,i} }}}\abs{\bar{h}^\epsilon_t(x)-h^{\epsilon}_t(x)}^p \right] \lesssim \epsilon^{p/2}
    \end{equation}
    where the constant involved in $\lesssim$ only depends on $p$, in particular it is uniform in $k, i $ and $\epsilon \in (0,1]$. Combining \eqref{découpage en sup sur les boites} and \eqref{estimate sup on space time box} yields
    \begin{equation*}
         \mathbb{E} \left[ \displaystyle{\sup_{\substack{t \in [0,T] \\ x \in [0,\infty)}}} \abs{\bar{g}^\epsilon_t(x)-g^{\epsilon}_t(x)}^p \right] \lesssim \lfloor \epsilon^{-2} T \rfloor \epsilon^{p/2}
    \end{equation*}
    where the constant involved in $\lesssim$ only depends on $p,\rho$ and $T$, in particular it is uniform in $\epsilon \in (0,1]$. The result follows for $p>4$, and then the result for all $p \geq 1$ is a direct consequence.
\end{proof}

\subsection{Proof of Theorem \ref{tightness of h}}

Recall that $C([0,\infty))$ denotes the set of all continuous functions on $[0,\infty)$ endowed with the topology of uniform convergence. We rely on the following tightness criterion (see for instance \cite[Section VII, Theorem 23.9]{kallenberg1997foundations})

\begin{proposition}(Tightness criterion){\large \textbf{.}} \label{tightness criterion} Let $(g^\epsilon)_{\epsilon \in (0,1]}$ be a family of $D([0,\infty),C([0,\infty)))$-valued random variables. Assume that
    \begin{enumerate}[(i)]
        \item For all $t \geq 0$, the family of $C([0,\infty))$ valued random variables $(g^\epsilon_t)_{\epsilon \in (0,1]}$ is tight.
        \item \label{second point tightness} $\forall T \geq 0 \qquad  \displaystyle{\lim_{\delta \to 0}} \: \displaystyle{\uplim_{\epsilon \to 0}} \; \mathbb{E} \left[ \displaystyle{\sup_{\substack{s,t \in [0,T] \\ \abs{s-t}<\delta}}} \norme{g^{\epsilon}_{t}-g^{\epsilon}_s}_{\infty} \right] =0$
    \end{enumerate}
    then $(g^\epsilon)_{\epsilon \in (0,1]}$ is tight in $D([0,T],C([0,\infty)))$, and any limit point belongs to $C([0,T],C([0,\infty)))$.
\end{proposition}

\noindent First, let us check that $(i)$ is satisfied. Because $h_0$ has law $\pi$, for all $t \geq 0$ we have $(g^\epsilon_t)_{\epsilon \in (0,1]} \overset{\mathcal{L}}{=} (g^\epsilon_0)_{\epsilon \in (0,1]}$. Now tightness is a direct consequence of the invariance principle for random walk conditioned to remain non-negative, see \cite[Theorem 2.1]{bryndoney}. Second, let us check that $(ii)$ is satisfied. To do so, we need to enhance our previous moment estimates to obtain uniformity in time. This can be achieved thanks to Kolmogorov's continuity lemma as stated in  \cite[Theorem I, 2.1]{revuz2013continuous}. Indeed, Kolmogorov's continuity lemma together with the estimate \eqref{estimate Cb norm} show that if we fix $b \in (0,\f{1}{2})$  and let $\theta \in (0,1)$ be the associated interpolation parameter given by Lemma \ref{lemme Cb}, then if we take $p$ large enough such that $\f{\kappa}{2}-\f{1}{2p}>0$, then for all $\alpha \in (0, \f{\kappa}{2}-\f{1}{2p})$ there exists a constant $c>0$ such that
\begin{equation} \label{estimate after kolmogorov}
    \displaystyle{\sup_{\epsilon \in (0,1]} \mathbb{E}} \left[ \left( \displaystyle{\sup_{\substack{t \neq s \\ s,t \in [0,T]}}} \f{ \norme{\bar{g}^\epsilon_t - \bar{g}^\epsilon_s}_{C^b} }{\abs{t-s}^\alpha} \right)^p \right] \leq c
\end{equation}
Now we can prove that hypothesis \eqref{second point tightness} of Proposition \ref{tightness criterion} is satisfied. We have for all $\delta, \epsilon \in (0,1]$
\begin{equation}\label{Eq:Holder}\begin{split}
    \mathbb{E} \left[ \displaystyle{\sup_{\substack{\abs{t-s}<\delta \\ s,t \in [0,T]}}} \norme{g^{\epsilon}_t-g^{\epsilon}_s}_{\infty} \right]
    &\leq \mathbb{E} \left[ \left( \displaystyle{\sup_{\substack{\abs{t-s}<\delta \\ s,t \in [0,T]}}} \norme{g^{\epsilon}_t-g^{\epsilon}_s}_{\infty} \right)^p \right]^{\f{1}{p}} \\
    &\leq 2  \mathbb{E} \left[ \displaystyle{\sup_{t \in [0,T]}} \norme{\bar{g}^\epsilon_t - g^\epsilon_t}_{\infty}^p \right]^{\f{1}{p}} +
    \mathbb{E} \left[ \left( \displaystyle{\sup_{\substack{t \neq s \\ s,t \in [0,T]}}} \; \f{ \norme{\bar{g}^\epsilon_t - \bar{g}^\epsilon_s}_{C^b} }{\abs{t-s}^{\alpha}} \right)^p \right]^{\f{1}{p}} \delta^\alpha
\end{split}
\end{equation}
now inequality \eqref{estimate after kolmogorov} together with Lemma \ref{estimation of the linearization error} conclude the proof of \eqref{second point tightness}.

\section{Convergence of \texorpdfstring{$(W^\epsilon)_{\epsilon \in (0,1]}$}{xi} to a cylindrical Wiener process}

\noindent For $\epsilon \in (0,1]$, recall that $(W^\epsilon_t)_{t \geq 0}$ is an $\mathcal{S}'([0,\infty))$-valued random process defined by
\begin{align}
    W_t^\epsilon(\varphi)&:= \Int{0}{t} \varphi(x) W_t^\epsilon (dx) \label{definition martingale} \\
    &=\f{\epsilon}{\sqrt{2}} \sum_{x \in \epsilon \N^*}{} \Int{0}{t} \Delta^\epsilon h^\epsilon_s(x) 1 \left\lbrace h_s^\epsilon (x) +\Delta^\epsilon h^\epsilon_s (x) \geq 0 \right\rbrace \varphi (x) \, dM_s^\epsilon(x) \nonumber
\end{align}
for all $t \ge 0$ and $\varphi \in \mathcal{S}([0,\infty))$. As defined, $W^\epsilon$ is a $D([0,\infty),\mathcal{S}'([0,\infty))$-valued random variable.
The main result of this section is the following
\begin{theorem} \label{convergence of xi}
The following convergence holds
\begin{equation}
    W^\epsilon \xrightarrow[\epsilon \to 0]{\mathcal{L}} W
\end{equation}
as $D([0,\infty),\mathcal{S}'([0,\infty))$-valued random variables, where $W$ is a cylindrical Wiener process.
\end{theorem}

\noindent To do so, we rely on the martingale structure of the dynamic. 

\subsection{Convergence of the bracket process}
In  this paragraph we prove the following result
\begin{proposition} \label{functional convergence of the bracket process}
    For any $\varphi \in \mathcal{S}([0,\infty))$, the following convergence holds in $D([0, \infty), \R)$
    \begin{equation} \label{functional cv of the bracket process}
        \langle\!\langle W^\epsilon(\varphi), W^\epsilon(\varphi) \rangle\!\rangle \xrightarrow[\epsilon \to 0]{\mathcal{L}} C
    \end{equation}
    where $C$ is the deterministic process defined by $C_t:= t \norme{\varphi}_{L^2([0,\infty))}$.
\end{proposition}

\noindent Let us decompose the bracket process $\langle\!\langle W^\epsilon (\varphi), W^\epsilon (\varphi) \rangle\!\rangle$ as follows
\begin{equation} \label{decomposition of the bracket process}
    \langle\!\langle W^\epsilon (\varphi) ,W^\epsilon (\varphi)  \rangle\!\rangle_t = A^{\epsilon,1}_t(\varphi) -A^{\epsilon,2}_t(\varphi) \qquad \forall t \geq 0
\end{equation}
with
\begin{align}
    A^{\epsilon,1}_t(\varphi) &:= \f{1}{2}  \sum_{x \in \epsilon \N^*}{} \Int{0}{t}
     \Delta^\epsilon h^\epsilon_s (x)^2 \varphi (x)^2 
    \: ds \\
    A^{\epsilon,2}_t(\varphi) &:=\f{1}{2}\sum_{x \in \epsilon \N^*}{} \Int{0}{t}
     \Delta^\epsilon h^\epsilon_s (x)^2
    1 \left\lbrace h_s^\epsilon (x) +\Delta^\epsilon h^\epsilon_s (x) < 0 \right\rbrace \varphi (x)^2 
    \: ds
\end{align}

\noindent Let us start by the two following lemmata, which contain the main ingredients to prove the convergence of the bracket process
\begin{lemma}(Returns to zero under the invariant measure){\large \textbf{.}} \label{returns to zero under the invariant measure} For every $\varphi \in \mathcal{S} ([0,\infty) )$ and $t \geq 0$
    \begin{equation}
            A^{\epsilon,2}_t(\varphi) \xrightarrow[\epsilon \to 0]{L^1(\mathbb{P})} 0
    \end{equation}
\end{lemma}

\begin{proof}
    We have
    \begin{align*}
        \E \left[ A^{\epsilon,2}_t(\varphi)  \right]
        &=\mathbb{E} \left[ \f{1}{2}  \sum_{x \in \epsilon \N^*}{} \Int{0}{t} (\Delta^\epsilon h_s^\epsilon)^2 1 \left\lbrace h_s^\epsilon (x) +\Delta^\epsilon h^\epsilon_s(x) < 0 \right\rbrace \varphi (x)^2 \: ds \right] \\
        &\leq 2 \epsilon \, \mathbb{E}  \left[ \sum_{x \in \epsilon \N^*}{} \Int{0}{t} 1 \left\lbrace h_s^\epsilon (x) +\Delta^\epsilon h^\epsilon_s(x) < 0 \right\rbrace \varphi (x)^2 \: ds \right] \\
        &\leq 2 \epsilon t \mathbb{E}  \left[ \sum_{x \in \epsilon \N^*}{} 1 \left\lbrace h_0^\epsilon (x) +\Delta^\epsilon h^\epsilon_0(x) < 0 \right\rbrace \varphi(x)^2 \right] \\
        &\leq 2 t \epsilon \pi \left[ \sum_{n \in \N^*  }{} 1 \left\lbrace X_n +\Delta X_n < 0 \right\rbrace \varphi (\epsilon n)^2 \right] \\
        &\xrightarrow[\epsilon \to 0]{} 0
    \end{align*}
    We used the fact that the process starts from the stationary measure $\pi$ in the third line, and used Lemma \ref{lemme 1 annexe} to obtain the last line.
\end{proof}

\begin{lemma}(Corners under the invariant measure){\large \textbf{.}} \label{corners under the invariant measure}
    For every $\varphi \in \mathcal{S}( [0,\infty))$ and  $t \geq 0$ 
    \begin{equation}
        \left( A^{\epsilon,1}_t(\varphi) - \epsilon t \sum_{x \in \epsilon \N^*}{} \varphi (x)^2 \right) \xrightarrow[\epsilon \to 0]{L^1(\mathbb{P})} 0
    \end{equation}
\end{lemma}

\begin{proof}
    We have
    \begin{align*}
     \E \left[ \left| A^{\epsilon,1}_t(\varphi) - \epsilon t \sum_{x \in \epsilon \N}{} \varphi (x)^2 \right| \right] 
    &= \mathbb{E}\left[ \f{1}{2} \left| \sum_{x \in \epsilon \N^*}{} \Int{0}{t} \left( \Delta^\epsilon h_s^\epsilon(x)^2  - 2\epsilon \right) \varphi(x)^2  \: ds \right| \right] \\
    &=\mathbb{E}\left[ \f{1}{2} \left|  \sum_{x \in \epsilon \N^*}{} \Int{0}{t} \left( 4\epsilon 1 \left\lbrace  \Delta^\epsilon h_s^\epsilon (x) \neq 0 \right\rbrace - 2\epsilon  \right) \varphi(x)^2  \: ds \right|  \right] \\
    &\leq 2 \epsilon \Int{0}{t}   \mathbb{E}  \left[ \left| \sum_{x \in \epsilon \N^* }{}  \left( 1 \left\lbrace  \Delta^\epsilon h_s^\epsilon (x) \neq 0 \right\rbrace -\f{1}{2}  \right) \varphi(x)^2 \right|  \right] \: ds \\
    &\leq 2 \epsilon t \pi  \left[ \left| \sum_{n \in \N^* }{}  \left( 1 \left\lbrace  \Delta X_n \neq 0 \right\rbrace -\f{1}{2}  \right) \varphi(\epsilon n)^2 \right|  \right] \\
    &\xrightarrow[\epsilon \to 0]{} 0
    \end{align*}
    We used the fact that the process starts from the stationary measure $\pi$ to obtain the fourth line, and Lemma \ref{lemme 2 annexe} to obtain the last line.
\end{proof}

\noindent We can now proceed to the proof of the convergence of the bracket process.

\begin{proof}[Proof of Proposition \ref{functional convergence of the bracket process}]
    Let $\varphi \in \mathcal{S}( [0,\infty))$ first, from Lemma \ref{corners under the invariant measure}, Lemma \ref{returns to zero under the invariant measure} and the decomposition \eqref{decomposition of the bracket process}, we deduce that for all $t \geq 0$
    \begin{equation*}
        \langle\!\langle W^\epsilon (\varphi), W^\epsilon (\varphi) \rangle\!\rangle_t \xrightarrow[\epsilon \to 0]{L^1(\mathbb{P})} t \norme{\varphi}_{L^2( [0,\infty))}^2
    \end{equation*}
    This in particular proves finite-dimensional convergence in law of  $\langle\!\langle W^\epsilon (\varphi), W^\epsilon (\varphi) \rangle\!\rangle$ towards $C$. Now since the processes $\langle\!\langle W^\epsilon (\varphi), W^\epsilon (\varphi) \rangle\!\rangle$ and $C$ are increasing and since $C$ is continuous, finite-dimensional convergence in law implies convergence in law by \cite[Theorem VI 3.37]{jacod2013limit}, which proves \eqref{functional cv of the bracket process}.
\end{proof}

\subsection{Convergence of the martingale} In this paragraph we fix a cylindrical Wiener process $W$ and prove the following convergence result

\begin{proposition} \label{functional convergence in law of the martingale}
    For all $\varphi \in \mathcal{S}([0,\infty))$, the following convergence holds
    \begin{equation} \label{functional cv in law of the martingale}
        W^\epsilon(\varphi) \xrightarrow[\epsilon \to 0]{\mathcal{L}} W(\varphi)
    \end{equation}
     as $D([0,\infty), \R)$-valued random variables.
\end{proposition}

\noindent To do so, we rely on the convergence of the bracket process proved in the previous paragraph, and on the convergence criterion \cite[Theorem VIII, 3.11]{jacod2013limit} which can be written in the following way: 

\begin{theorem} \label{functional limit theorem jacod shiryaev}
    Let $(X^\epsilon)_{\epsilon \in (0,1]}$ be a family of càdlàg martingales and $X$ a continuous Gaussian martingale. Assume that
    \begin{enumerate}[(i)]
        \item There exists $K>0$ such that almost surely, 
        \begin{equation*}
            \forall t \ge 0, \forall \epsilon \in (0,1] \qquad \abs{X^\epsilon_t - X^\epsilon_{t^-}} \le K.
        \end{equation*}
        \item The following convergence holds
        \begin{equation*}
            \displaystyle{\sup_{t \ge 0} \abs{X^\epsilon_t - X^\epsilon_{t^-}}} \xrightarrow[\epsilon \to 0]{\mathbb{P}} 0.
        \end{equation*}
        \item The following convergence of $D([0,\infty),\R)$-valued random variables holds
        \begin{equation*}
            \langle\!\langle X^\epsilon, X^\epsilon \rangle\!\rangle \xrightarrow[\epsilon \to 0]{\mathcal{L}} \langle\!\langle X, X \rangle\!\rangle.
        \end{equation*}
        
    \end{enumerate}
    Then $X^\epsilon \xrightarrow[\epsilon \to 0]{\mathcal{L}} X$ as $D([0,\infty),\R)$-valued random variables.
\end{theorem}

\noindent We now prove the convergence result.

\begin{proof}[Proof of Proposition \ref{functional convergence in law of the martingale}]
    Let $\varphi \in \mathcal{S}([0,\infty))$. It suffices to apply Theorem \ref{functional limit theorem jacod shiryaev} to the family of martingales $(W^\epsilon(\varphi))_{\epsilon \in (0,1]}$ and to the continuous Gaussian martingale $W(\varphi)$, whose bracket process is $C$. It follows from the deterministic bound 
    \begin{equation*}
        \forall t \ge 0, \forall \epsilon \in (0,1], \qquad \abs{W^\epsilon_t(\varphi)-W^\epsilon_{t^-}(\varphi)} \leq \sqrt{2} \epsilon^{3/2} \| \varphi\|_\infty
    \end{equation*}
    that assumptions $(i)$ and $(ii)$ of Theorem \ref{functional limit theorem jacod shiryaev} are satisfied. Moreover, assumption $(iii)$ of Theorem \ref{functional limit theorem jacod shiryaev} is a consequence of Proposition \ref{functional convergence of the bracket process}. Thus by Theorem \ref{functional limit theorem jacod shiryaev} we deduce that \eqref{functional cv in law of the martingale} is satisfied, which concludes the proof.
\end{proof}

\subsection{Convergence towards white noise}
In order to prove Theorem \ref{convergence of xi}, we rely on the following tightness criterion in $D \left([0,\infty),\mathcal{S}'([0,\infty)) \right)$.

\begin{lemma}(Mitoma's criterion \cite{mitoma}){\large \textbf{.}} \label{Mitoma}
A collection $(X^\epsilon)_{\epsilon \in (0,1]}$ of $D \left([0,\infty),\mathcal{S}'([0,\infty)) \right)$-valued random variables is tight if and only if for all $\varphi \in \mathcal{S}([0,\infty))$, the collection $(X^\epsilon(\varphi))_{\epsilon \in (0,1]}$ is tight in $D([0,\infty),\R)$.
\end{lemma}

\noindent We now prove the main result of this section.

\begin{proof}[Proof of Theorem \ref{convergence of xi}]
    It follows from Proposition \ref{functional convergence in law of the martingale} and Lemma \ref{Mitoma} that the sequence $(W^\epsilon)_{\epsilon \in (0,1]}$ is tight in $D([0,\infty),\mathcal{S}'([0,\infty)))$. By Le Cam's generalization of Prokhorov's theorem (see \cite[Theorem 6.7]{walsh} or \cite[Section 5, Theorem 2]{smolyanov1976measures}) the collection $(W^\epsilon)_{\epsilon \in (0,1]}$ is relatively sequentially compact for the convergence in law. Let $X$ be a $D([0,\infty),\mathcal{S}'([0,\infty)))$-valued random variable which is a limit point of $(W^\epsilon)_{\epsilon \in (0,1]}$. Then for all $\varphi \in \mathcal{S}'([0,\infty))$, by continuity of the map 
    \begin{equation*}
        \begin{array}{rl}
            \pi_\varphi : D([0,\infty),\mathcal{S}'([0,\infty))) &\longrightarrow  D([0,\infty), \R ) \\
            (x_t)_{t \geq 0} &\longmapsto (x_t(\varphi))_{t \geq 0}
        \end{array}
    \end{equation*}
    we have that
    \begin{equation} \label{cv in law test against varphi}
        W^\epsilon (\varphi) \xrightarrow[\epsilon \to 0]{\mathcal{L}} X(\varphi)
    \end{equation}
    Now from \eqref{cv in law test against varphi} and Proposition \ref{functional convergence in law of the martingale} we deduce that $X(\varphi)$ is a Brownian Motion of variance $\norme{\varphi}_{L^2([0,\infty))}$. Thus, $X$ is a cylindrical Wiener process and we have characterized uniquely the law of any limit point, which concludes the proof.
\end{proof}

\section{Tightness of \texorpdfstring{$(\eta^\epsilon)_{\epsilon \in (0,1]}$}{eta} }

\noindent In this section we focus on the discrete reflection term
\begin{equation}
    \eta^\epsilon(dt,dx):= \f{2}{\sqrt{\epsilon}} \sum_{k \in \epsilon \N}{} 1 \left\lbrace h_t^\epsilon (k) +\Delta^\epsilon h_t^\epsilon (k) <0 \right\rbrace \delta_k (dx) \: dt.
\end{equation}
We endowed the set $\mathbb{M}$ defined in \eqref{space M of measures}, with the smallest topology that makes
\begin{equation*}
         \nu \in \mathbb{M} \mapsto \Int{[0,\infty) \times [0,\infty)}{} x \psi (t,x) \nu(dt,dx)
\end{equation*}
continuous, for all maps $\psi \in C_c([0,\infty) \times [0,\infty))$. Our goal in this section is to prove that the collection $(\eta^\epsilon)_{\epsilon \in (0,1]}$ is tight in $\mathbb{M}$. 
Roughly speaking, thanks to the semi-discrete PDE satisfied by $h^\epsilon$, tightness of $(\eta^\epsilon)_{\epsilon \in (0,1]}$ will be a consequence of tightness of $(h^\epsilon)_{\epsilon \in (0,1]}$ and of $(W^\epsilon)_{\epsilon \in (0,1]}$. To make this rigorous, we first need to control an error term due to the discretization $\langle \cdot , \cdot \rangle_{\epsilon}$ of the $L^2([0,\infty),dx)$ inner product $\langle \cdot, \cdot \rangle$ 
\begin{align*}
R^\epsilon_t(\varphi)&:=\langle h_t^\epsilon, \varphi \rangle_\epsilon - \langle h_t^\epsilon, \varphi \rangle
        -\langle h^\epsilon_0, \varphi \rangle_\epsilon + \langle h^\epsilon_0, \varphi \rangle \\
        &- \epsilon \sum_{x \in \epsilon \N}{} \Int{0}{t} \f{1}{\epsilon^2} \Delta^\epsilon h^\epsilon_s (x) \varphi (x) \: ds +  \Int{0}{t} \langle h^\epsilon_s,  \varphi ''  \rangle \: ds.
\end{align*}
With this definition, the semi-discrete PDE \eqref{variational semimartingale equation} satisfied by $h^\epsilon$ rewrites
\begin{multline} \label{semi discrete pde with error term}
    \langle h^\epsilon_t, \varphi \rangle -\langle h^\epsilon_0, \varphi \rangle 
    - \Int{0}{t} \langle h^\epsilon_s, \varphi '' \rangle \: ds
    - \Int{[0,t]\times [0,\infty)}{} \varphi(x) \: d\eta^\epsilon(ds,dx) - \sqrt{2} W^\epsilon_t(\varphi) + R^\epsilon_t(\varphi) =0
\end{multline}
where $W^\epsilon_t(\varphi)$ was introduced in \eqref{definition martingale}.
The next lemma shows that the error term vanishes in law as $\epsilon \to 0$.
\begin{lemma} \label{control of the error term}
    Let $\varphi \in C^\infty_c( [0,\infty))$ such that $\varphi(0)=0$. Then
    \begin{equation}
        (R_t^\epsilon(\varphi),t\ge 0) \xrightarrow[\epsilon \to 0]{\mathcal{L}} 0\;.
    \end{equation}
\end{lemma}
\noindent We postpone the proof of this lemma to Appendix \ref{proof control error term}. We can now state and prove the main result of this section.

\begin{theorem} \label{tightness of eta}
    The collection $(\eta^\epsilon)_{\epsilon \in (0,1]}$ of $\mathbb{M}$-valued random variables is tight. Furthermore any limit point $\eta$ satisfies the following property almost surely: for all $\varphi \in C^\infty_c( [0,\infty))$ such that $\varphi(0)=0$, $t\mapsto \int_{[0,t] \times [0,\infty)} \varphi(x) \: \eta(ds,dx)$ is continuous on $[0,\infty)$.
\end{theorem}

\begin{proof} By the tightness criterion from Lemma \ref{tightness criterion random measures} in Appendix, it suffices to prove that for all $t \ge 0$ and $\varphi$ such that $\varphi (x)=x f(x)$ for all $x \geq 0$ with $f \in C^\infty_c( [0,\infty))$, the sequence of real valued random variables
	\begin{equation*}
		\left( \Int{[0,\infty) \times [0,\infty)}{} \varphi(x) \: \eta^\epsilon(ds,dx) \right)_{\epsilon \in (0,1]}
	\end{equation*}
	is tight.
	For $\epsilon \in (0,1]$ and $\varphi$ as above, by the semi-discrete PDE \eqref{semi discrete pde with error term}
	\begin{multline}
	\Int{[0,t] \times [0,\infty)}{} \varphi(x) \: d\eta^\epsilon(ds,dx) = \langle h^\epsilon_t, \varphi \rangle -\langle h^\epsilon_0, \varphi \rangle  
	- \Int{0}{t} \langle h^\epsilon_s, \varphi ''  \rangle \: ds
	- \sqrt{2} W^\epsilon_t(\varphi) + R^\epsilon_t(\varphi)
	\end{multline}
	By Theorem \ref{tightness of h}, the first three terms are tight. Tightness of $(R_{t}^\epsilon(\varphi))_{\epsilon \in (0,1]}$ is a consequence of Lemma \ref{control of the error term} while tightness of $(W_t^\epsilon(\varphi))_{\epsilon \in (0,1]}$ is a consequence of Lemma \ref{functional convergence in law of the martingale}. We thus deduce that the l.h.s.~is tight and this concludes the proof of the first part of the statement.\\
	We now turn to the second part of the statement. Let $\eta$ be the limit of a converging subsequence: for simplicity, we still write $(\eta^\epsilon)_{\epsilon}$ this subsequence. By an approximation argument, it is sufficient to prove that for any given $\varphi \in C^\infty_c([0,\infty))$ which satisfies $\varphi(0) = 0$ and is \emph{non-negative}, almost surely $t\mapsto \int_{[0,t] \times [0,\infty)} \varphi(x) \: \eta(ds,dx)$ is continuous on $[0,\infty)$. Therefore we fix such a $\varphi$ until the end of the proof. Let us define
	$$ X^{\epsilon}_t := \int_{[0,t] \times [0,\infty)} \varphi(x) \: \eta^\epsilon(ds,dx)\;,\quad t\ge 0\;.$$
	The arguments above actually showed that $(X^{\epsilon}_t, t\ge 0)_{\epsilon \in (0,1]}$ is $C$-tight. Up to an extraction, we can thus assume that $(\eta^\epsilon,X^\epsilon)$ converges in law to $(\eta,X)$ where $X$ is continuous.\\
	For any $0 \le a \le b$, let $\chi_{a,b}:\R\to [0,1]$ be a smooth function satisfying:
	$$ 1{\{[0,a]\}}(t) \le \chi_{a,b}(t) \le  1{\{[0,b]\}}(t)\;,\quad t\ge 0\;.$$
	We now write for any $t\ge 0$ and $\delta > 0$ (small enough)
	$$ X^\epsilon_{t-2\delta} \le \int_{[0,\infty) \times [0,\infty)} \chi_{t-2\delta,t-\delta}(s) \varphi(x) \: \eta^\epsilon(ds,dx) \le \int_{[0,\infty) \times [0,\infty)} \chi_{t+\delta,t+2\delta}(s) \varphi(x) \: \eta^\epsilon(ds,dx) \le X^\epsilon_{t+2\delta}\;.$$
	Passing to the limit along the subsequence we obtain
	$$ X_{t-2\delta} \le \int_{[0,\infty) \times [0,\infty)} \chi_{t-2\delta,t-\delta}(s) \varphi(x) \: \eta(ds,dx) \le \int_{[0,\infty) \times [0,\infty)} \chi_{t+\delta,t+2\delta}(s) \varphi(x) \: \eta(ds,dx) \le X_{t+2\delta}\;.$$
	Now observe that
	\begin{align*}
		X_{t-2\delta} \le \int_{[0,\infty) \times [0,\infty)} \chi_{t-2\delta,t-\delta}(s) \varphi(x) \: \eta(ds,dx) &\le \int_{[0,t] \times [0,\infty)} \varphi(x) \: \eta(ds,dx)\\
		&\le \int_{[0,\infty) \times [0,\infty)} \chi_{t+\delta,t+2\delta}(s) \varphi(x) \: \eta(ds,dx) \le X_{t+2\delta}\;,
	\end{align*}
	so that, passing to the limit $\delta \downarrow 0$, we deduce that for any $t\ge 0$, almost surely
	$$ X_t = \int_{[0,t] \times [0,\infty)} \varphi(x) \: \eta(ds,dx)\;.$$
	Since $X$ is continuous, and the process on the r.h.s.~is càdlàg, this equality holds almost surely for all $t\ge 0$, and the asserted continuity follows.
\end{proof}

\section{Proof of the main results}

\begin{proof}[Proof of Theorem \ref{main theorem}]
	By Theorems \ref{tightness of h}, \ref{convergence of xi} and \ref{tightness of eta}, we know that $(h^\epsilon, W^\epsilon, \eta^\epsilon)_{\epsilon \in (0,1]}$ is tight in $D([0,\infty), \mathcal{C}_\rho) \times D([0,\infty), \mathcal{S}'([0,\infty) ) \times \mathbb{M}$. We have to identify the law of the limit points. Let $(u,W,\eta)$ be the limit of a converging subsequence. In order to alleviate the notations, we write $(h^\epsilon, W^\epsilon, \eta^\epsilon)_{\epsilon \in (0,1]}$ the subsequence. By Theorem \ref{convergence of xi}, we already know that $W$ is a cylindrical Wiener process. We will check that $(u,\eta)$ satisfy the conditions listed in Definition \ref{definition de l'eq reflechie}, and will conclude using the strong uniqueness for this stochastic PDE. Items (i), (ii) and (iii) are automatically satisfied by any elements of our spaces. Let us check the last two items.
	
	\paragraph{Item (iv) - Limiting equation}
	Fix $t > 0$ and $\varphi \in C^\infty_c([0,\infty))$ such that $\varphi(0)=0$. Consider the map
	\begin{equation*}
		\begin{array}{ll}
			F_{t,\varphi}: &D([0,\infty), \mathcal{C}_\rho) \times D([0,\infty), \mathcal{S}'([0,\infty) ) \times \mathbb{M}  \longrightarrow \R  \\
			&(h,V,\nu) \longmapsto
			\langle h_t, \varphi  \rangle -\langle h_0, \varphi \rangle
			- \Int{0}{t} \langle h_s , \varphi'' \rangle \: ds 
			 - V_t(\varphi)
			 - \Int{0}{t} \Int{0}{\infty} \varphi(s,x) \: \nu (ds,dx).
		\end{array}  
	\end{equation*}
	Consider also the space
        \begin{equation*}
            \tilde{\mathbb{M}}:= \left\lbrace \nu \in \mathbb{M} \: : \: \forall t \in [0,\infty) \quad \nu \left( \{ t \} \times (0,\infty) \right) = 0 \right\rbrace.
        \end{equation*}
        Then $F$ restricted to $C([0,\infty), \mathcal{C}_\rho) \times C([0,\infty), \mathcal{S}'([0,\infty) ) \times \tilde{\mathbb{M}}$ is continuous. Additionally, it follows from Theorems \ref{tightness of h}, \ref{convergence of xi} and \ref{tightness of eta} that the law of $(u,W,\eta)$ is concentrated on $C([0,\infty), \mathcal{C}_\rho) \times C([0,\infty), \mathcal{S}'([0,\infty) ) \times \tilde{\mathbb{M}}$. Therefore, by the continuous mapping theorem
	\begin{equation} \label{limit eq 1}
		F_{t,\varphi} (h^\epsilon,W^\epsilon, \eta^\epsilon) \xrightarrow[\epsilon \to 0]{\mathcal{L}} F(u,W , \eta)
	\end{equation}
	On the other hand, the semi-discrete PDE \eqref{variational semimartingale equation} tells us that for all $\epsilon \in        (0,1]$
	\begin{equation*}
		F_{t,\varphi}(h^\epsilon, W^\epsilon, \eta^\epsilon) +R_t^\epsilon(\varphi)=0
	\end{equation*}
	but by Lemma \ref{control of the error term}, $R_t^\epsilon(\varphi) \xrightarrow[\epsilon \to 0]{\mathcal{L}} 0$, so by        Slutsky's theorem
	\begin{equation} \label{limit eq 2}
		F_{t,\varphi}(h^\epsilon, W^\epsilon, \eta^\epsilon) \xrightarrow[\epsilon \to 0]{} 0
	\end{equation}
	Now by \eqref{limit eq 1} and \eqref{limit eq 2} and uniqueness of the limit we deduce that 
	\begin{equation}
		F_{t, \varphi} (u, W, \eta) = 0
	\end{equation}
	which concludes the proof of (iv).
	
	\paragraph{Item (v) - Support condition}
	Let $\psi \in C^\infty_c([0,\infty)\times [0,\infty))$ be a non-negative function, and let $T$ be such that $\textrm{supp }(\psi) \subseteq [0,T] \times [0,\infty)$. Consider the map
	\begin{equation*}
		\begin{array}{rl}
			F :{D}([0,\infty),\mathcal{C}_\rho) \times \mathbb{M}& \longrightarrow \R \\
			(h,m)&\longmapsto \Int{[0,\infty) \times [0,\infty)}{} x\psi(s,x) h(s,x) \: m (ds,dx).
		\end{array}
	\end{equation*}
	Then $F$ restricted to ${C}([0,\infty),\mathcal{C}_\rho) \times \mathbb{M}$ is continuous with respect to the product topology. Additionnally, it follows from Theorem \ref{tightness of h} that the law of $(u,\eta)$ is concentrated on ${C}([0,\infty),\mathcal{C}_\rho) \times \mathbb{M}$. Consequently, by the continuous mapping theorem
	\begin{equation} \label{1 pour support condition}
		F(h^\epsilon,\eta^\epsilon) \xrightarrow[\epsilon \to 0]{\mathcal{L}} F(u,\eta)\;.
	\end{equation}
	But we also have
	\begin{equation} \label{encadrement pour support condition}
		0 \leq F(h^\epsilon,\eta^\epsilon) = \sqrt{\epsilon} \Int{[0,\infty) \times [0,\infty)}{} x \psi (s,x) \eta^\epsilon (ds,dx)
	\end{equation}
	since by definition of the discrete reflection measure $\eta^\epsilon$, $h^\epsilon$ is equal to $\sqrt{\epsilon}$ on the support of $\eta^\epsilon$. Because the sequence $\left( \Int{[0,\infty) \times [0,\infty)}{} x\psi (s,x) \eta^\epsilon (ds,dx) \right)_{\epsilon \in (0,1]}$ converges in law towards an almost-surely finite random variable, we obtain that the right hand side of \eqref{encadrement pour support condition} converges in distribution towards zero. Thus, we have
	\begin{equation} \label{2 pour support condition}
		F(h^\epsilon,\eta^\epsilon) \xrightarrow[\epsilon \to 0]{\mathcal{L}} 0
	\end{equation}
	From \eqref{1 pour support condition} and \eqref{2 pour support condition} we deduce that for all non-negative $\psi \in C^\infty_c([0,\infty)\times [0,\infty))$,
	\begin{equation}
		\Int{[0,\infty) \times [0,\infty)}{} x\psi(s,x) u(s,x) \: \eta (ds,dx) =0
	\end{equation}
	almost surely. By the Monotone Convergence Theorem, this suffices to deduce that almost surely
	$$\Int{[0,\infty) \times [0,\infty)}{} u(s,x) \: \eta (ds,dx) =0\;,$$
	concluding the proof of (v).
\end{proof}

\begin{proof}[Proof of Corollary \ref{Corollary}]
	From the convergence of Theorem \ref{main theorem} and since $h^\epsilon$ is stationary with law $\pi^\epsilon$, we deduce that $u$ is a solution of \eqref{Nualar Pardoux eq} which is stationary. At each time $t$, the law of $u(t,\cdot)$ is the limit of the laws $\pi^\epsilon$, which by~\cite{bryndoney} is nothing but the law of the $3$-dimensional Bessel process.
\end{proof}


\appendix

\section{Piecewise linear interpolation on \texorpdfstring{$\epsilon \N$}{TEXT}}
\subsection{Fourier transform}
\noindent Let $g: \epsilon \N \xrightarrow[]{} \R$ such that $g(0)=0$ and let us still write $g$ for its piecewise linear interpolation in space, which we assume integrable. By definition of the Fourier transform, for $\zeta \in \R$
\begin{equation}
    \hat{g}(\zeta)= \Int{[0,\infty)}{} g(x) e^{-i \zeta x} \: dx
\end{equation}
The following lemma gives an expression which is simply a convenient rewriting for the Fourier transform, leveraging the fact that $g$ is piecewise affine.
\begin{lemma} \label{lemma expression fourrier coefficients}
    For any $\zeta \in \R$, and any $\epsilon \in (0,1]$ we have
    \begin{equation}
    \hat{g}(\zeta) = c_{\zeta,\epsilon} \sum_{n \in \epsilon \N}{} e^{-i\zeta n} g(n)
\end{equation}
with 
\begin{equation} \label{def and bound for coef c}
    c_{\zeta,\epsilon}:= \f{2}{\epsilon \zeta^2} (1- \cos (\zeta \epsilon)) \in [0,\epsilon]
\end{equation}
\end{lemma}

\begin{proof} We have
\begin{align*}
    \hat{g}(\zeta) &= \Int{\R}{} g(x) e^{-i \zeta x} \: dx \\
    &=\sum_{n \in \epsilon \N}{} \; \Int{n}{n+\epsilon} g(x) e^{-i\zeta x} \: dx \\
    &=\epsilon \sum_{n \in \epsilon \N} \Int{0}{1} g(n+\lambda \epsilon) e^{-i \zeta (n+\lambda \epsilon)} \: d \lambda \\
    &=\epsilon \sum_{n \in \epsilon \N} e^{-i \zeta n} \Int{0}{1} \left[ g(n)+ \lambda (g(n+\epsilon)-g(n)) \right] e^{-i \zeta \lambda \epsilon} \: d \lambda \\
\end{align*}
    Set $a_{\zeta,\epsilon}:=\Int{0}{1} e^{-i \zeta \lambda \epsilon} \: d \lambda$ and $b_{\zeta,\epsilon}:=\Int{0}{1} \lambda e^{-i \zeta \lambda \epsilon} \: d \lambda $, then
\begin{align*}
    \hat{g}(\zeta) &=\epsilon \sum_{n \in \epsilon \N} e^{-i \zeta n} \left[  a_{\zeta ,\epsilon} g(n) + b_{\zeta ,\epsilon} \left( g(n+\epsilon)-g(n) \right) \right] \\
    &=\epsilon \sum_{n \in \epsilon \N} e^{-i \zeta n} \left[ a_{\zeta ,\epsilon} + b_{\zeta ,\epsilon} (e^{i \zeta \epsilon } -1)  \right] g(n) \\    
\end{align*}
Moreover, a direct computation yields $a_{\zeta, \epsilon}=\f{i e^{-i \zeta \epsilon} -i}{ \zeta \epsilon} $ and $b_{\zeta ,\epsilon}=\f{i e^{-i\zeta \epsilon}}{\zeta \epsilon} + \f{e^{-i \zeta \epsilon}-1}{\zeta^2 \epsilon^2}$.
Consequently, setting $c_{\zeta,\epsilon} :=\epsilon ( a_{\zeta ,\epsilon} + b_{\zeta ,\epsilon} (e^{i \zeta \epsilon } -1) )$ yields the result.
\end{proof}

\subsection{Proof of Lemma \ref{control of the error term}} \label{proof control error term}

\begin{lemma} For every $A >0$ and every $T \geq 0$
    \begin{equation} \label{uniform bound expectation of the norm sup of h}
        \displaystyle{\sup_{\epsilon \in (0,1] } \mathbb{E} \left[ \displaystyle{\sup_{t \in [0,T]}} \norme{h^\epsilon_t}_{\infty, [0,A]} \right]} < \infty
    \end{equation}
\end{lemma}

\begin{proof}
    This is a direct consequence of \eqref{Eq:Holder} and of the estimates \eqref{Eq:akbk} that ensure, with the help of Kolmogorov Continuity Theorem, that the moments of $\norme{h^\epsilon}_{\infty, [0,A]}$ under $(\pi^\epsilon)_{\epsilon \in (0,1]}$ are uniformly bounded in $\epsilon$.
\end{proof}

\begin{lemma} \label{error terms 1 2 3}For every $\varphi \in C^\infty_c( [0,\infty))$ such that $\varphi(0) = 0$ and $T \geq 0$
    \begin{enumerate}[(i)]
    \item $\mathbb{E} \left[ \displaystyle{\sup_{t \in [0,T]}}\abs{\langle h^\epsilon_t,\varphi \rangle_\epsilon - \langle h^\epsilon_t,\varphi \rangle} \right] \xrightarrow[\epsilon \to 0]{} 0$ \\
    \item $\mathbb{E} \left[ \displaystyle{\sup_{t \in [0,T]}}\abs{\epsilon \sum_{x \in \epsilon \N}{} \Int{0}{t} \f{1}{\epsilon^2} \Delta^\epsilon h^\epsilon_s (x) \varphi (x) \: ds -  \Int{0}{t} \langle h^\epsilon_s, \varphi '' \rangle \: ds} \right] \xrightarrow[\epsilon \to 0]{} 0$
    \end{enumerate}
\end{lemma}
\begin{proof} Let $\varphi \in C^\infty_c([0,\infty))$, such that $\varphi(0)=0$ and $A>0$ such that $\textrm{supp}(\varphi) \subseteq [0,A]$. Using the fact that $h_t^\epsilon$ is piecewise affine on the lattice $\epsilon \N$, we have
    \begin{align*}
        &\abs{\langle h^\epsilon_t,\varphi \rangle_\epsilon - \langle h^\epsilon_t,\varphi \rangle} \\
        &=\abs{\Int{0}{\infty} h^\epsilon_t(y) \varphi(y) \: dy - \epsilon \sum_{x \in \epsilon \N}{} h_t^\epsilon(x) \varphi(x) } \\
        &=\abs{\sum_{x \in \epsilon \N}{} h^\epsilon_t(x) \left( \Int{x}{x+\epsilon} \f{x+\epsilon-y}{\epsilon} \varphi (y) \: dy + \Int{x-\epsilon}{x} \f{y-x+\epsilon}{\epsilon} \varphi (y) \: dy - \epsilon \varphi(x)  \right)} \\
        &=\abs{\sum_{x \in \epsilon \N}{} h^\epsilon_t(x) \left( \Int{x}{x+\epsilon}\f{x+\epsilon-y}{\epsilon} (\varphi (y) - \varphi(x)) \: dy + \Int{x-\epsilon}{x} \f{y-x+\epsilon}{\epsilon}  (\varphi (y)- \varphi (x)) \: dy  \right)} \\
        &\leq  \epsilon A \norme{h^\epsilon_t}_{\infty, [0,A]} \norme{ \varphi '}_{\infty} 
    \end{align*}
    where we used the mean value theorem in the last line. This enables us to conclude for $(i)$ using \eqref{uniform bound expectation of the norm sup of h}.
    Let us turn to the proof of $(ii)$. For fixed $s \in [0,t]$, we have
    \begin{align*}
        &\abs{\epsilon \sum_{x \in \epsilon \N}{} \f{1}{\epsilon^2} \Delta^\epsilon h^\epsilon_s (x) \varphi (x) -  \langle h^\epsilon_s,  \varphi '' \rangle} \\
        & \quad = \abs{ \Int{0}{\infty} h^\epsilon_s(y)  \varphi '' (y) \: dy - \epsilon \sum_{x \in \epsilon \N}{}  h^\epsilon_s(x) \f{1}{\epsilon^2}  \Delta^\epsilon \varphi (x)} \\
        & \quad =\abs{\sum_{x \in \epsilon \N}{} h^\epsilon_s(x) \left( \Int{x}{x+\epsilon} \f{x+\epsilon-y}{\epsilon} \varphi '' (y) \: dy + \Int{x-\epsilon}{x} \f{y-x+\epsilon}{\epsilon} \varphi '' (y) \: dy  -  \epsilon \f{1}{\epsilon^2}  \Delta^\epsilon \varphi (x) \right)} \\
        &\quad = \bigg|\sum_{x \in \epsilon \N}{} h^\epsilon_s(x) \Big( \Int{x}{x+\epsilon} \f{x+\epsilon-y}{\epsilon} (\varphi '' (y)- \f{1}{\epsilon^2} \Delta^\epsilon \varphi(x)) \: dy
        	+ \Int{x-\epsilon}{x} \f{y-x+\epsilon}{\epsilon}  ( \varphi '' (y)-\f{1}{\epsilon^2} \Delta^\epsilon \varphi(x))  \: dy \Big) \bigg| \\
        &\quad \leq \epsilon A \norme{h^\epsilon_t}_{\infty,[0,A]} \norme{ \varphi ^{(3)}}_{\infty, [0,A]}.
    \end{align*}
    The last line is obtained thanks to the mean value theorem. As for the first point, \eqref{uniform bound expectation of the norm sup of h} enables us to conclude.
\end{proof}
\noindent Eventually, Lemma \ref{control of the error term} is a consequence of Lemma \ref{error terms 1 2 3}.

\section{Tightness criterion for random measures}

\begin{lemma} \label{tightness criterion random measures}
    Let $(\eta^\epsilon)_{\epsilon \in (0,1]}$ be a family of random elements of $\mathbb{M}$ and assume that for all $f \in C^\infty_c( [0,\infty))$, and $t \ge 0$ the family of real valued random variables $\left( \Int{[0,t] \times [0,\infty)}{} x f(x) \: \eta^\epsilon (dt,dx) \right)_{\epsilon \in (0,1]}$ is tight. Then the family $(\eta^\epsilon)_{\epsilon \in (0,1]}$ is tight in $\mathbb{M}$.
\end{lemma}

\begin{proof}
	First, notice that under our assumption on $(\eta^\epsilon)_{\epsilon \in (0,1])}$, for every $\psi \in C^\infty_c([0,\infty)\times [0,\infty))$, the family $\left( \Int{[0,t] \times [0,\infty)}{} x \psi(t,x) \: \eta^\epsilon (dt,dx) \right)_{\epsilon \in (0,1]}$ is tight. Indeed, let $\psi \in  C^\infty_c([0,\infty)\times [0,\infty))$. Then let $A,T>0$ such that $\textrm{supp}(\psi) \subseteq [0,T]\times [0,A]$, and take $f \in  C^\infty_c([0,\infty))$ such that $1_{[0,A]} \leq f$. From the inequality
	\begin{equation*}
		\abs{\int_{[0,\infty)\times[0,\infty)} x \psi(t,x) \: \eta^\epsilon (dt,dx)} \leq \norme{\psi}_{\infty} \int_{[0,T]\times [0,\infty)} x f(x) \eta^\epsilon (dt,dx)
	\end{equation*}
	since the right hand side of the inequality is tight by assumption, we get that the left hand side is tight as well.
    Second, let us turn now to the proof of the tightness. Taking a family $(\psi_k)_{k \in \N}\in C^\infty_c([0,\infty) \times [0,\infty))$ which is dense in $C_c([0,\infty) \times [0,\infty))$ for the uniform topology and letting $\varphi_k(t,x):=x \psi_k(t,x)$, we have that
    \begin{equation}
        d(\eta,\eta'):=\sum_{k \in \N} 2^{-k} \left( 1 \wedge \abs{\Int{}{} \varphi_k \: d \eta - \Int{}{} \varphi_k \: d \eta'} \right)
    \end{equation}
    defines a metric compatible with the topology on $\mathbb{M}$. Observe that by sequential extraction, for any sequence $\lambda \in [0,\infty)^\N$, the set
    \begin{equation*}
        A_\lambda:= \left\lbrace \eta \in \mathbb{M} \: : \: \forall k \in \N \; \abs{\Int{}{} \varphi_k \: d\eta} \leq \lambda_k \right\rbrace
    \end{equation*}
    is relatively compact in $\mathbb{M}$. Let $\delta >0$. By assumption on $(\eta^\epsilon)_{\epsilon \in (0,1]}$, for any $ k \in \N$ there exists $\lambda_k \in [0,\infty)$ such that
    \begin{equation*}
        \displaystyle{\sup_{\epsilon \in (0,1]}} \mathbb{P} \left( \abs{\Int{}{} \varphi_k \: d\eta^\epsilon } > \lambda_k \right) < \delta 2^{-k}.
    \end{equation*}
    We deduce by subadditivity that
    \begin{equation*}
        \displaystyle{\sup_{\epsilon \in (0,1]}} \mathbb{P} \left( \exists k \in \N, \; \abs{\Int{}{} \varphi_k \: d\eta^\epsilon } > \lambda_k \right) < \delta.
    \end{equation*}
    In other words,
    \begin{equation*}
        \displaystyle{\sup_{\epsilon \in (0,1]}} \mathbb{P} \left( \eta^\epsilon \notin A_\lambda \right) < \delta
    \end{equation*}
    since $A_\lambda$ is relatively compact, this concludes the proof.
\end{proof}

%

\bibliographystyle{siam}
\bibliography{sample}

\end{document}